\documentclass[review]{elsarticle}

\usepackage{lineno,hyperref}
\modulolinenumbers[5]
\usepackage{multirow}
\usepackage{makerobust}
\MakeRobustCommand\overrightarrow
\MakeRobustCommand\overleftarrow
\usepackage{pdflscape}
\usepackage{endnotes}
\usepackage{amsmath}
\usepackage{graphicx}
\usepackage{caption}
\usepackage{subcaption}
\usepackage{amsmath, amsthm, amssymb}
\usepackage{appendix}
\usepackage{graphicx}
\usepackage{resizegather}
\usepackage{color}
\usepackage{xcolor}
\usepackage{arydshln}
\usepackage{chngcntr}
\usepackage{apptools}
\usepackage{float}

\newtheorem{proposition}{Proposition}
\newtheorem{corollary}{Corollary}
\newtheorem{lemma}{Lemma}
\newtheorem{definition}{Definition}
\newtheorem{remark}{Remark}
\AtAppendix{\counterwithin{remark}{section}}
\AtAppendix{\counterwithin{proposition}{section}}

\begin{document}

\begin{frontmatter}
\title{An opaque selling scheme to reduce shortage and wastage in perishable inventory systems}


\author[mymainaddress]{Katsunobu Sasanuma\corref{mycorrespondingauthor}}
\ead{katsunobu.sasanuma@stonybrook.edu}
\cortext[mycorrespondingauthor]{Corresponding author}

\author[mysecondaryaddress]{Akira Hibiki}
\ead{hibiki@tohoku.ac.jp}

\author[mymainaddress]{Thomas Sexton}
\ead{thomas.sexton@stonybrook.edu}

\address[mymainaddress]{College of Business, Stony Brook University, Stony Brook, NY 11794, USA}
\address[mysecondaryaddress]{Graduate School of Economics and Management, Tohoku University, Sendai, Japan}

\begin{abstract}
Effective management of perishable inventory systems is often strewn with challenges, especially when a strong trade-off relationship exists between shortage and wastage of perishables: A smaller inventory increases the chance to lose sales (leading to higher expected shortage cost), while a larger inventory increases the chance to waste perishables (leading to higher expected wastage cost). The root cause of this strong trade-off relationship is high product demand variability. To mitigate the issue and reduce the cost of operating perishable inventory systems, some grocery stores utilize an opaque selling scheme: selling an anonymous product whose brand or exact specification is shielded from customers at the time of sale. The use of opaque products has now become a popular means to reduce shortage/wastage at grocery stores. \textcolor{black}{However, there has been little discussion of the effectiveness of opaque schemes applied to perishable inventory systems.}

\textcolor{black}{In this paper, we propose an opaque scheme based on the balancing policy on demand, which tries to average out orders for products. We confirm both analytically and numerically that this opaque scheme effectively reduces product demand variability, thereby reducing both shortage and wastage for perishable inventory systems under a base-stock policy. We also present an analytical formula that reveals insights into the opaque scheme: The ratio between the opaque proportion and the coefficient of variation of product demands plays a key role to determine the effectiveness of our opaque scheme. Furthermore, we provide a rule of thumb to find the threshold opaque scheme parameters needed to achieve the maximum total cost savings for perishable inventory systems. We hope that many retailers selling perishables (e.g., fresh produce and baked goods) find the opaque scheme useful, implement it, and contribute to the reduction of the food wastage.}
\end{abstract}


\end{frontmatter}


\section{Introduction}
In 2011, about one-third of food produced for human consumption was wasted globally \cite{gustavsson2011global}. In the United States in 2010, about 133 billion pounds (31\%) of the food supply was wasted, including 43 billion pounds at retail level \cite{buzby2014estimated}. This large amount of food waste not only imposes an economic loss for producers and sellers, but also exacerbates food insecurity for low-income communities \cite{munesue2015effects}. One of the contributors to food waste is poor management of perishable inventory: \cite{buzby2014estimated} point out that food waste at retail level occurs from ``overstocking or overpreparing due to difficulty predicting number of customers.'' \textcolor{black}{As is commonly known, ``businesses are frequently motivated to overproduce or over-order because they are afraid of running out of food or failing to meet client expectations'' \cite{brodersen2019oregon}.} For retail store managers, overstocking perishable products is necessary in anticipation of the possibility of lost sales when product demand is highly variable. However, overstocking leads to higher expected wastage. The root cause of this trade-off relationship is high variability of product demand. If the product demand was accurately predictable with no uncertainty, a seller would order the exact amount necessary and yield no lost sales (i.e., shortage) or wastage; however, an accurate prediction of product demand is not always possible.

To cope with highly variable product demand in perishable inventory systems, retail store owners often utilize an opaque selling scheme (or simply, an \emph{opaque scheme}). This scheme involves the sales of an opaque product, wherein customers only know the types of products they may possibly get, but do not know the exact product attributes until they purchase or receive it. The simplest example would be the following. Suppose a seller receives an order of 30 red apples, 50 yellow apples, and 20 \emph{opaque} apples; the seller allocates 20 opaque orders to red apples, making demands of both red and yellow apples equal to 50, so that the seller can average out red and yellow apples' demands and achieve less variability for both product demands. Many merchants and producers implement a variety of opaque schemes successfully without calling them \emph{opaque}: for example, a package of assorted food items (such as bagels and donuts), whose combination is not easily identified in a brown bag, is often sold at a discount price at supermarkets; Apple Japan sells Lucky Bags that contain various currently-sold merchandises in sealed bags at a discount price; Priceline sells anonymous hotels as ``Express Deal'' together with hotels that are fully specified. Another example is the app, ``Too Good To Go'', which aims to reduce food waste by selling unspecified food items (e.g., bakeries, delis, and cakes). \textcolor{black}{The number of mobile apps that implement various opaque schemes has rapidly grown in recent years. In addition, the COVID-19 pandemic has driven customers to shift from physical to online stores across many categories such as fresh produce and groceries, making the implementation of an opaque scheme easier than before.}

\subsection{\textcolor{black}{Contributions and outline}}
\textcolor{black}{The benefit of opaque schemes for non-perishables (e.g., clothes) has been observed in practice and discussed in the literature. However, even though we observe opaque schemes being utilized at many grocery stores and bakeries selling fresh foods through mobile apps, a quantitative analysis of opaque schemes for perishables has not been published. This paper introduces an opaque scheme suitable for perishables, examines its effectiveness for perishable inventory systems, and provides simple rules of thumb for practitioners when implementing our opaque scheme.}

In this paper, we first review the previous literature and explain the research gap in \S\ref{sec2}. We provide approximate formulas to obtain the variability (Corollary~\ref{relativevariance}) and correlation (Lemma \ref{involution}) of product demands under our opaque scheme in \S\ref{sec3}; we provide an approximate formula to obtain the threshold variability to make the expected shortage and wastage close to zero (Proposition \ref{wastage}) in \S\ref{sec4}; we confirm the accuracy of the analytical formulas using numerical experiments in \S\ref{sec:numerical}; and we discuss the managerial insights into the opaque scheme in \S\ref{sec6}. Finally, \S\ref{sec7} summarizes our study. The list of notations used in this paper and all proofs are in the appendix.

\section{Related Literature}
\label{sec2}
To reduce the amount of wastage, many retail store owners often sell their about-to-expire products at significantly discounted prices. This practice may not be desirable; knowledge that there may be a last-minute discount could dissuade buyers from purchasing items at regular (higher) prices. Reinforcing such a buyer habit may decrease the average selling price and seriously damage the market \cite{sviokla2003value}. Another common practice is to donate older products to food banks, which certainly contributes to the reduction of wastage; however, retail stores still pay for wastage. The root cause, high variability of product demand, still remains. 

\subsection{\textcolor{black}{Key idea: risk pooling}}
\textcolor{black}{To cope with risk (volatile and unpredictable customer demand), operations managers frequently use a scheme called \emph{risk pooling}---a strategy to suppress the variability using pooled demand (combined demand). Risk pooling has been proven effective in the context of supply chain management (see, e.g., \cite{davis1993effective, graves2003process, simchi2012understanding}). An opaque scheme is one such example: It utilizes pooled demand as the means to control highly variable and unpredictable demands of products. This scheme has attracted much attention especially in revenue management (\cite{jerath2010revenue, anderson2012choice}). For example, Priceline has reduced shortage (lost sales) and wastage of hotel rooms by offering opaque hotel rooms to their customers. Under Priceline's opaque scheme, room buyers do not know the brand of the hotel prior to purchase; however, they may reserve an opaque hotel room (at a discount room rate) that would stay unoccupied and be wasted if Priceline did not offer opaque hotel rooms. \cite{anderson2009setting} reveals the relationship between opaque hotel room pricing and hotel room inventory, thereby confirming the effectiveness of Priceline's opaque scheme.}

\subsection{\textcolor{black}{Research gap}}
\textcolor{black}{Although an increasing attention is paid to opaque schemes in revenue management, the literature on opaque schemes in inventory systems is still very small. \cite{elmachtoub2021power, elmachtoub2015retailing, elmachtoub2019value} study an effectiveness of opaque schemes in non-perishable inventory systems. \cite{elmachtoub2021power} focus on the pricing schemes of opaque products when pessimistic and risk-neutral customers exist. \cite{elmachtoub2015retailing} and \cite{elmachtoub2019value} discuss the opaque scheme for non-perishable inventory systems following a continuous-review $(0, S)$ policy, where an order is placed to make the inventory level back to the order-up-to level $S$ when the inventory is depleted. An opaque scheme is applied to the system with two non-opaque products \cite{elmachtoub2015retailing} or more \cite{elmachtoub2019value}; their analysis is limited to the no lead time, backlogging, or lost sales case.}

\textcolor{black}{These existing studies focus on non-perishable products (e.g., clothes), not on perishable products (e.g., fresh produce, baked goods). As far as authors know, theoretical analysis of opaque schemes has not been made for perishable inventory systems, despite the fact that many retail stores already implement such schemes. Our study fills the research gap by looking into the effectiveness of opaque schemes in perishable inventory systems. We apply the no lead time assumption for our study in accord with previous opaque scheme studies. Our paper has two major differences from existing studies on non-perishable inventory systems: (1) we consider a periodic base-stock policy, which is more suitable for perishable products; and (2) our opaque order fulfillment procedure is based on the number of orders received in each period (BPD) and not based on the on-hand inventory (BPI).}

\subsection{\textcolor{black}{BPI (balancing policy on inventory)}}
To carry out the risk pooling principle in a perishable inventory system, we can consider \emph{balancing policy on inventory} (BPI): Using pooled demand (opaque product demand), BPI attempts to average out the inventory level among products. The actual procedure to implement BPI would be to ``fulfill demand for opaque products by using the product with the highest on-hand inventory level'' \cite{elmachtoub2019value}. For non-perishable inventory systems, BPI is often considered as an optimal policy \cite{mcgavin1997balancing}; furthermore, even a small proportion of opaque product demand can effectively improve the operation \cite{elmachtoub2019value}. However, sellers may face challenges when implementing BPI in perishable inventory systems:
First, larger on-hand inventory often implies larger amount of aged (about-to-perish) inventory. Thus, BPI may implicitly prioritize the use of older units to fulfill opaque demand whenever possible. When sellers implement BPI, customers might regard opaque products as not only \emph{opaque}, but also \emph{aged}. (For example, Lucky Bags sold at Apple Japan, a popular opaque selling scheme, seldom contain recently-released products. Instead, we often observe a bag full of obsolete products since sellers implement BPI.) Such a perception harms the reputation of stores and provides a disincentive for customers to purchase opaque products unless they are sold at hugely discounted prices. Second, the proper execution of BPI for perishable inventory is extremely complicated; for example, it is hard to determine which inventory to use first, smaller inventory of older products v.s. larger inventory of newer products. These challenges do not exist for non-perishables, but they may become a major obstacle when implementing BPI for perishables.

\subsection{\textcolor{black}{BPD (balancing policy on demand)}}
\textcolor{black}{We propose a simpler alternative, \emph{balancing policy on demand} (BPD), which attempts to average out product demand.} BPD fulfills demand for opaque products by using the product with the least demand. BPD does not utilize on-hand inventory information, and thus may have inferior performance compared to BPI; however, we show that BPD only requires a small amount of opaque product demand to average out demands among products, thereby achieving a balanced inventory level that BPI also achieves. Furthermore, BPD does not suffer from two challenges that BPI faces: First, BPD is simpler to implement in practice than BPI; moreover, under BPD, a seller does not necessarily allocate an older inventory to fulfill the demand of opaque products. BPD requires a purchaser of opaque products to give up only the \emph{specificity} of products (in exchange for some monetary or nonmonetary benefits for customers), not the \emph{quality} (i.e., freshness, age) of products. \textcolor{black}{We show in this paper that our opaque scheme using BPD effectively reduces the variability of demands and achieves high cost savings for perishable inventory systems.}

\subsection{\textcolor{black}{Scaled Poisson demand for non-opaque products}}
\textcolor{black}{BPD opaque schemes (as well as BPI) can be applicable to any product demand distributions (regardless of perishables or non-perishables), but for the sake of analytical studies, we assume that a periodic demand of each product follows a scaled Poisson distribution.} An alternative, more commonly used demand distribution in inventory models would be a compound Poisson distribution\footnote{A compound Poisson distribution is versatile and can replicate various consumer purchasing behaviors: Each customer arrives following a Poisson process, whose purchases follow an arbitrary distribution.} (see, e.g., \cite{dominey2004performance, babai2011analysis} and references therein); however, we utilize a scaled Poisson distribution since it can accurately approximate various compound Poisson distributions \cite{bohm2014statistics} and yet, is simpler to implement than a compound Poisson distribution. Throughout this study, we allow demands and orders for inventory replenishment to be non-integers (note: examples of non-integer perishables are cut vegetables/fruits, which are gaining in popularity recently).

\subsection{\textcolor{black}{Base-stock model for perishable inventory systems}}
\textcolor{black}{The optimal inventory policy for a perishable inventory system is hard to find due to \emph{curse of dimensionality}---the intractability of all possible change of inventory items with different ages over infinite period. Thus, various heuristic methods have been proposed (for a review, please see \cite{karaesmen2011managing, baron2011managing}; one of the examples is \cite{sasanuma2020marginal}). Among many heuristic policies that are available today, a simple base-stock policy is still considered as an effective scheme \cite{huh2009asymptotic, bijvank2014robustness, bu2019asymptotic}. In this study, we use this simple base-stock policy. We further impose simplifying assumptions used in \cite{nahmias1973optimal}: fixed shelflife, i.i.d. (independent and identically distributed) demand, periodic review, FIFO (first-in first-out) issuance policy, and the total cost that involves shortage and wastage costs. The analysis of the models with more relaxed assumptions would be the next step after this study.}

\section{Opaque Selling Scheme}
\label{sec3}
\textcolor{black}{This section introduces our $(n,p)$ opaque scheme based on BPD policy. The result we obtain in this section applies to both perishable and non-perishable products. The goal of the $(n,p)$ opaque scheme is to reduce the variability of non-opaque product demands by allocating opaque product demand to non-opaque product demands (i.e., fulfilling opaque product orders using non-opaque products.) Notations used in this study are summarized in Appendix~\ref{notation}}

\subsection{\textcolor{black}{Model}}
The $(n,p)$ opaque scheme is characterized by~${n\; (\ge2)}$ non-opaque products and probability vector~${\boldsymbol{p}:=(p^1,p^2,\cdots,p^n)}$, where each~${p^i \in [0,1]}$ represents the probability that a customer for non-opaque product~${i \in \{1,\cdots,n\}}$ switches to purchase an opaque product if both opaque and non-opaque products are offered to customers. We introduce three random variables to describe the demands at various stages: original demands ($D_0^i$ for non-opaque product $i$), intermediate demands ($X_{\boldsymbol{p}}^i$ for non-opaque product $i$ and $X_{\boldsymbol{p}}^0$ for an opaque product), and adjusted demands ($D_{\boldsymbol{p}}^i$ for non-opaque product $i$). \emph{Original demands} are the number of orders placed for non-opaque products when no opaque product is offered to customers. We represent original demands by independent (but not necessarily identical) random variable $D_0^i$ for product $i$, whose mean is $\mu^i=\mathbb{E}[D^i_0]$ and variance is $\sigma_i^2=\mathbb{E}[(D^i_0-\mu^i)^2]$.  \emph{Intermediate demands} are the number of orders placed for non-opaque and opaque products when both non-opaque/opaque products are offered to customers. We represent intermediate demands by $X_{\boldsymbol{p}}^i$ for non-opaque product $i$ and $X_{\boldsymbol{p}}^0$ for an opaque product. Finally, \emph{adjusted demands} are the number of orders for non-opaque products including some of the orders originally placed for an opaque product. (This adjustment is essentially the procedure to distribute opaque demand $X_{\boldsymbol{p}}^0$ to $n$ non-opaque demands $X_{\boldsymbol{p}}^i, \forall i$.)

\textcolor{black}{Using these variables, we can describe the procedure to implement our opaque scheme under BPD:\\
\textbf{Step 1}: (status quo) All products sold are non-opaque, and their demands are $D_0^i$ (original demand).\\
\textbf{Step 2}: (implementation) An opaque product is now sold. All non-opaque product demands change from $D_0^i$ to $X_{\boldsymbol{p}}^i$ (intermediate demand). In addition, an opaque product shows its demand $X_{\boldsymbol{p}}^0$.\\
\textbf{Step 3}: (adjustment) The entire opaque demand $X_{\boldsymbol{p}}^0$ is distributed and added to $X_{\boldsymbol{p}}^i, \forall i$. Through this adjustment, we obtain $D_{\boldsymbol{p}}^i$ (adjusted demand).}

\subsection{\textcolor{black}{Assumptions}}
In order to facilitate the analysis, we make a few simplifying assumptions for the model. First, we require the average of adjusted demand (orders) remain unchanged: $\mathbb{E}[D_{0}^i]=\mathbb{E}[D_{\boldsymbol{p}}^i]\; (=\mu^i), \forall \boldsymbol{p}, \forall i$. Second, we assume that the total demand (total number of orders) remains unchanged after an opaque product is offered to customers:
$\sum_i D^i_0 = \sum_i X_{\boldsymbol{p}}^i + X^0_{\boldsymbol{p}}=\sum_i D_{\boldsymbol{p}}^i, \forall \boldsymbol{p}.$
Lastly, we assume that the customers' decisions are independent from each other: ${\mathbb{E}[X_{\boldsymbol{p}}^i]=(1-p^i)\mu^i}$ and ${\mathbb{E}[X_{\boldsymbol{p}}^0]=\sum_i p^i \mu^i}$. Thus, we can interpret $p^i$ as a proportion of orders shifted from non-opaque product $i$ to an opaque product under the opaque scheme: ${p^i=1-\mathbb{E}[X_{\boldsymbol{p}}^i]/\mu^i}$. A seller can control the proportion $\boldsymbol{p}$ by providing customers incentives to purchase an opaque product. Examples of such incentives include discounts, reduced/free shipping fees, and labeling an opaque product as an eco-friendly choice that helps reduce food waste. If no customers are inclined to purchase an opaque product, then $\boldsymbol{p}=\boldsymbol{0}=(0,0, \cdots,0)$ (no opaque case), and thus $X^0_{\boldsymbol{0}}=0$ and $D^i_0=X_{\boldsymbol{0}}^i=D_{\boldsymbol{0}}^i, \forall i$. In contrast, if all customers always purchase an opaque product, then
$\boldsymbol{p}=\boldsymbol{1}=(1,1, \cdots,1)$ (full opaque case), and thus,  $X^i_{\boldsymbol{1}}=0, \forall i$ and $\sum_i D^i_0=X^0_{\boldsymbol{1}}=\sum_i D^i_{\boldsymbol{1}}$. For simplicity of notation, if all elements of vector $\boldsymbol{p}$ are equivalent, we denote a subscript as scalar $p$ unless it creates a confusion (e.g., we denote $D^i_0$ instead of $D_{\boldsymbol{0}}^i$).

\subsection{\textcolor{black}{Variance $\sigma_{n,\boldsymbol{p}}^2$}}
Let the average variance of original demands be ${\sigma^2}:=\frac{1}{n}\sum_i var(D_0^i)$ and the average variance of adjusted demands under the $(n,p)$ opaque scheme be ${\sigma_{n,{\boldsymbol{p}}}^2}:=\frac{1}{n}\sum_i var(D_{\boldsymbol{p}}^i)$. The following lemma shows the bound of $\sigma_{n,\boldsymbol{p}}^2$.

\begin{lemma}
\label{bound}
The average variance $\sigma_{n,\boldsymbol{p}}^2$ is bounded as follows. $\sigma_{n,\boldsymbol{p}}^2$ takes an upper bound $\sigma_{n,{\boldsymbol{0}}}^2=\sigma^2$ when $\boldsymbol{p}=\boldsymbol{0}$ and takes a lower bound $\sigma_{n,{\boldsymbol{1}}}^2=\frac{\sigma^2}{n}$ when $\boldsymbol{p}=\boldsymbol{1}$.
\small
\begin{equation}
\label{pa1}
\sigma^2 \geq \sigma_{n,{\boldsymbol{p}}}^2 \geq \frac{\sigma^2}{n}.
\end{equation}
\normalsize
\end{lemma}

Lemma~\ref{bound} implies that $\sigma_{n,{\boldsymbol{p}}}^2$ is minimized at $p=1$ when the deviation of each product demand from its mean is averaged out among all products: ${D_1^i=\mu^i+\frac{1}{n} \sum_j(D_0^j-\mu^j)}$. We can establish the opaque product demand allocation policy (BPD) as Corollary~\ref{policy}. We omit the proof, which follows straightforwardly from the fact that $\sigma_{n,\boldsymbol{p}}^2$ is minimized when we minimize the term $\mathop{\sum \sum}_{i<j} \left((D^i_p-\mu^i)-(D^j_p-\mu^j)\right)^2$ (see the proof of Lemma~\ref{pa1}).

\begin{corollary}
\label{policy}
Balancing policy on demand (BPD): To minimize the average variance $\sigma_{n,\boldsymbol{p}}^2$ under the $(n,p)$ opaque scheme, we distribute opaque product demand ($X^0_{\boldsymbol{p}}$) to non-opaque product demands ($X^i_{\boldsymbol{p}}, \forall i \in \{1,\cdots,n\})$ with smallest $X_{\boldsymbol{p}}^i-\mu^i$ or distribute equally if there is a tie.
\end{corollary}

For the rest of this paper, we impose additional conditions to obtain analytical expressions for some performance indices. Specifically, we assume constant~$p$ for all products and an i.i.d. scaled Poisson random variable $D_0$ to represent original non-opaque product demands. To obtain the distribution of $D_0$, we first start from a Poisson random variable with parameter $\lambda$ (which we call \emph{base Poisson parameter} of a scaled Poisson random variable) and scale its distribution to make its mean equal to~$\mu$. Specifically, denoting $Y_\gamma \sim Pois(\lambda)$, we can represent $D_0$ as $D_0=\frac{\mu}{\lambda}Y_{\lambda}$. It follows that the fully adjusted demand $D_1$ under the $(n,1)$ opaque scheme (i.e., $D_p$ with~$p=1$; a sample mean of~$n$ i.i.d.~$D_0$) also follows a scaled Poisson distribution: $D_1=\frac{\mu}{n\lambda}Y_{n\lambda}$.\footnote{Note that $D_p$ with $p\in (0,1)$ and $n \neq 1$ is not scaled Poisson; it is a complex mixture of two scaled Poisson distributions.} For use in Proposition~\ref{wastage}, we also introduce a sample mean of i.i.d. $D_1$ over $m$ shelflife periods: $\overline{D}_1=\frac{\mu}{nm\lambda}Y_{nm\lambda}$. For ease of reference, we summarize the properties of these scaled Poisson demands in Table~\ref{tab:originaldemand}.
\begin{table}[h]
\begin{center}
\caption{Properties of scaled Poisson demands $D_0$, $D_1$, and $\overline{D}_1$.}
\footnotesize
\begin{tabular}{c|cccc}
\noalign{\smallskip} \hline
 demand  & base Poisson parameter        & mean           & variance& $c_v$  \\ \hline
 \rule{0pt}{4ex}  
$D_0$ &$\lambda$  & $\mu$ & $\sigma_{n,0}^2=\dfrac{\mu^2}{\lambda}\left(=:\sigma^2\right)$& $\dfrac{1}{\sqrt\lambda}$ \\ \rule{0pt}{4ex}  
$D_1$ &$n\lambda$  & $\mu$ & $\sigma_{n,1}^2=\dfrac{\mu^2}{n\lambda}\left(=\dfrac{\sigma^2}{n}\right)$& $\dfrac{1}{\sqrt{n\lambda}}$ \\ \rule{0pt}{4ex}  
$\overline{D}_1$ &$nm\lambda$  & $\mu$ & $\dfrac{\mu^2}{nm\lambda}\left(=\dfrac{\sigma^2}{nm}\right)$& $\dfrac{1}{\sqrt{nm\lambda}}$ \\ \hline
\end{tabular}
\label{tab:originaldemand}
\end{center}
\normalsize
\end{table}

\subsection{\textcolor{black}{Relative variance $\sigma_{rel}^2(p)$}}
\textcolor{black}{To find $\sigma_{n,p}^2$, it is convenient to use a new indicator, \emph{relative variance} $\sigma_{rel}^2(p)$. It represents the proximity of $\sigma_{n,p}^2$ to the sample mean variance; see Definition~\ref{relative}. Notice that the minimization of $\sigma_{rel}^2(p)$ is essentially equivalent to minimizing $\sigma_{n,p}^2(p)$. From now on, we will consider $\sigma_{rel}^2(p)$.}

\begin{definition}
\label{relative}
For all $n \ge 2$ and $p \in [0,1]$, relative variance under the $(n,p)$ opaque scheme is defined as
\small
\begin{equation}
\label{eq:definition}
\sigma_{rel}^2(p):=\frac{\sigma_{n,p}^2-\sigma_{n,1}^2}{\sigma_{n,0}^2-\sigma_{n,1}^2}
=\frac{\sigma_{n,p}^2-\frac{\sigma^2}{n}}{\sigma^2-\frac{\sigma^2}{n}}.
\end{equation}
\normalsize
\end{definition}
\textcolor{black}{Note that $\sigma_{n,p}^2$ takes $\sigma_{n,0}^2=\sigma^2$ when $p=0$ and takes $\sigma_{n,1}^2=\sigma^2/n$ when $p=1$ (see Lemma~\ref{bound}). By definition, we observe $\sigma_{rel}^2(p) \in [0,1]$, $\sigma_{rel}^2(1) =1$, and $\sigma_{rel}^2(0) =0$. If we know $\sigma_{rel}^2(p)$ (which we simply denote as $\sigma_{rel}^2$ unless there is confusion), we can derive $\sigma_{n,p}^2$ as
\small
\begin{equation}
\label{eq:relativevariance}
{\sigma_{n,p}^2=\frac{1+(n-1)\sigma_{rel}^2}{n} \sigma^2}.
\end{equation}
\normalsize}

\textcolor{black}{The new indicator $\sigma_{rel}^2$ is introduced primarily for mathematical convenience; however, $\sigma_{rel}^2$ has a relationship with a popular indicator, an average correlation coefficient $\rho_p$ of adjusted demands. We observe that a reduction of $\sigma_{rel}^2 \;{(=\sigma_{rel}^2(p))}$ makes all adjusted demands closer to their sample mean, and therefore, increases an average correlation coefficient $\rho_p$ of adjusted demands. For example, it is expected that $\sigma_{rel}^2(0)=1$ and $\rho_0=0$ at $p=0$, and $\sigma_{rel}^2(1)=0$ and $\rho_1=1$ at $p=1$. The following lemma tells us the relationship between $\sigma_{rel}^2$ and $\rho_p$ for any $p \in [0,1]$.}
\begin{lemma}
\label{involution}
Consider the $(n,p)$ opaque scheme, where $n \ge 2$. Relative variance and correlation coefficient satisfy the following relationships.
\small
\begin{equation*}
\sigma_{rel}^2(p) =\frac{1-\rho_p}{1+(n-1)\rho_p},
\end{equation*}
\normalsize
or equivalently,
\small
\begin{equation}
\label{rho}
\rho_p =\frac{1-\sigma_{rel}^2(p)}{1+(n-1)\sigma_{rel}^2(p)}.
\normalsize
\end{equation}
\end{lemma}

\begin{remark}
Lemma~\ref{involution} indicates that $\sigma_{rel}^2(p)$ and $\rho_p$ satisfy the involution property: The inverse operation and the forward operation are the same.
\end{remark}

\subsection{\textcolor{black}{Exact and approximate representations of $\sigma_{rel}^2(p)$: $n=2$ case}}
We now consider the $n=2$ case and find the exact analytical expression for $\sigma_{rel}^2$. First, we note the following proposition.
\begin{proposition}
\label{varepsilon}
Consider the $(n,p)$ opaque scheme with $n=2$. Define a random variable ${T=X_p^1-X_p^2-X^0_p}$ and its variance $\sigma_T=var(T)$. Relative variance of adjusted demand is exactly represented as a function of~$T$:
\begin{equation}
\label{exact}
\sigma_{rel}^2=\frac{2\mathbb{E}[T^2|T>0]Pr(T>0)}{\sigma_T^2}.
\end{equation}
\end{proposition}

\textcolor{black}{Equation~\eqref{exact} in Proposition~\ref{varepsilon} is exact, but is not convenient when we evaluate $\sigma_{rel}^2$. We show an approximation in Corollary~\ref{relativevariance}.}
\begin{corollary}
\label{relativevariance}
Let $D_0$ follow a scaled Poisson distribution with $c_v=1/\sqrt{\lambda}$. Denote $\alpha=p\sqrt{2\lambda}=\frac{p}{c_v}\sqrt{2}$. Using the standard normal CDF $\Phi(\cdot)$ and PDF $\phi(\cdot)$, we obtain the following approximation.
\begin{equation}
\label{sigma-rel}
\sigma_{rel}^2 \approx 2(1+\alpha^2)\Phi(-\alpha)-2\alpha\phi(\alpha).
\end{equation}
\end{corollary}

Equation~\eqref{sigma-rel} shows that relative variance $\sigma_{rel}^2$ is a function of $\alpha\; (\propto p/c_v)$; thus, we know that adjusted variance $\sigma_{n,p}^2$ does not change if both $p$ and $c_v$ proportionally change. This property provides a useful tip for practitioners when implementing our opaque scheme. Specifically, if we observe a change of $c_v$ in product demands, we should change $p$ proportionally in order to maintain the same level of $\sigma_{n,p}^2$ (see Figure~\ref{fig:p-cv-a5} and the discussion therein).  Note that Equation~\eqref{sigma-rel} is obtained for $n=2$; however, it can be used to obtain an approximate value of $\sigma_{rel}^2$ for any $n$ since $\sigma_{rel}^2$ is not sensitive to $n$ according to our numerical experiment (see Figure~\ref{fig:p-npr-a2} and the discussion therein).

\section{Perishable Inventory System}
\label{sec4}
In this section, we analyze the effect of our opaque scheme on the total cost of a single product lost-sales perishable inventory system \cite{nahmias1973optimal}. As in the previous sections, original demands of all non-opaque products are represented by an i.i.d. scaled Poisson random variable $D_0$ with mean ${\mu\: (=\mathbb{E}[D])}$ and variance ${\sigma^2\:(=\mu^2/\lambda)}$. We assume all orders are placed at the start of each period to make the inventory level equal to the base-stock level $q$ units; all ordered items arrive instantaneously and are new; and inventory is depleted following a FIFO issuance rule. If any orders are not fulfilled due to a lack of inventory, they are lost and incur a per-unit shortage cost $r$ (sales price minus purchase cost). If any products in inventory are not sold within $m$ shelflife periods, they are discarded and incur a per-unit wastage cost~$\theta$ (purchase cost plus recycling/landfill/compost cost minus salvage value of an outdated product). At the end of business hours, a seller implements an $(n,p)$ opaque scheme with BPD (i.e., fulfilling opaque orders using non-opaque products with least orders).

Let $S$ and $W$ be the number of shortages (lost sales) and wastages per period per product, respectively. We define the total cost as $C=rS+\theta W$. \textcolor{black}{We consider that the total holding cost is fixed with negligible per-product holding cost, and omit it from the total cost.} Our goal is to minimize the expected total cost $\mathbb{E}[C]$ utilizing the $(n,p)$ opaque scheme with BPD.

Denote CDF and PMF of a discrete random variable $Y$ as $F_{Y}(\cdot)$ and $P_{Y}(\cdot)$, respectively. Table~\ref{tab:originaldemand} summarizes the definitions of three scaled Poisson random demands we use: $D_0$, $D_1$, and $\overline{D}_1$. Let $s=\left \lfloor{\frac{n\lambda q}{\mu}}\right \rfloor$ and  $Y_\gamma \sim Pois(\gamma)$. We obtain the following proposition.
\begin{proposition}
\label{wastage}
Consider implementing the $(n,1)$ opaque scheme for perishable products with $m$-period shelflife.  The expected per-period per-product shortages, wastages, and total cost satisfy that, for any base-stock level $q\ge0$,
\small
\begin{equation}
\label{expshortage}
\mathbb{E}[S] = \mathbb{E}\left[D_1-q\right]^+=\left(\mu-q\right)(1-F_{Y_{n\lambda}}(s))+\mu P_{Y_{n\lambda}}(s),
\end{equation}
\begin{equation}
\label{expwastage}
m \mathbb{E}\left[\frac{q}{m}-\overline{D}_1\right]^+ \geq\mathbb{E}[W] \geq \mathbb{E}\left[\frac{q}{m}-\overline{D}_1\right]^+=\left(\frac{q}{m}-\mu \right)F_{Y_{nm\lambda}}(s)+\mu P_{Y_{nm\lambda}}(s).
\end{equation}
\normalsize
Thus, denoting \small $\overline{C}_{LB}=r\mathbb{E}\left[D_1-q\right]^+ + \theta \mathbb{E}\left[\frac{q}{m}-\overline{D}_1\right]^+$, $\overline{C}_{LB}\le \mathbb{E}[C]\le m \overline{C}_{LB}$ \normalsize holds.
\end{proposition}

According to Proposition~\ref{wastage}, $\mathbb{E}[C]\approx 0$ holds at a wider range of $q$ as $n \lambda$ increases since $F_{Y_{n\lambda}}(s)$ and $F_{Y_{nm\lambda}}(s)$ approach a step function at $q=\mu$ and $q=m\mu$, respectively, and both $P_{Y_{n\lambda}}(s)$ and $P_{Y_{nm\lambda}}(s)$ shrink (note: we observe this property experimentally in Figures~\ref{fig:totalcost} and \ref{fig:lam}). At the limit of $n \lambda \rightarrow \infty$, we observe $D_{1}, \overline{D}_1 \rightarrow \mu$ and $E[C] \rightarrow 0$ for any $q \in[\mu,m\mu]$. This insensitivity property was first pointed out by \cite{chazan1977markovian}; the property is important for practitioners who prefer their inventory systems to be robust.

\textcolor{black}{Proposition~\ref{wastage} can be used to find the condition to achieve $\mathbb{E}[C]\approx 0$. For this purpose, we introduce threshold parameters $n_{\text{th}}$ and ${\sigma_{\text{th}}^2}$ as Definition~\ref{threshold}. These parameters provide useful information for practitioners trying to achieve ${\mathbb{E}[C] \le m\delta}$ (where $\delta$ is sufficiently small; see Table~\ref{tab:lowerbound} and the discussion therein).
\begin{definition}
\label{threshold}
For sufficiently small $\delta$, the threshold variance~${\sigma_{\text{th}}^2}$ is defined as
\small
\begin{equation}
\label{eq:threshold}
\sigma^2_{\text{th}}=\max_{n \ge 2}\{\sigma^2_{n,1}\:|\:\overline{C}_\text{LB}\leq \delta\}.
\end{equation}
\normalsize
Also define the threshold number of non-opaque products as $n_{\text{th}}=n$ s.t. ${\sigma^2_{n,1}=\sigma^2_{\text{th}}}$.
\end{definition}
}

\section{Numerical Experiments}
\label{sec:numerical}
\subsection{\textcolor{black}{Simulation settings}}
\textcolor{black}{In this section, we numerically examine the effectiveness of our $(n,p)$ opaque scheme. We split the experiments into two parts. In the first part (\S\ref{5-2} and \S\ref{5-3}), we obtain a series of simulated adjusted demands $D_p$ under the opaque scheme, examine them, and evaluate the impact of the opaque scheme on the variance of adjusted demands $\sigma_{n,p}^2$ (or $\sigma_{rel}^2$); in the second part (\S\ref{5-4} to \S\ref{5-6}), we use the simulated $D_p$ to run a Monte Carlo simulation and evaluate the impact of the opaque scheme on the expected per-period total cost~$\mathbb{E}[C]$.}

We control variables as follows: the number of non-opaque products ${n \in [1,12]}$, proportion ${p\in[0,1]}$, base Poisson parameter ${\lambda\in\{4,6,8,10,12,14\}}$ which correspond to coefficient of variation $c_v\in\{0.50, 0.41, 0.35, 0.32, 0.29, 0.27\}$, and average demand $\mu=10$. We first generate 10,000 (sample values of) original demand $D_0$ for each one of the $n$ products. We then generate 10,000 sets of intermediate demands $X_p^0$ and $X_p^i, i \in \{1,\cdots,n\}$, from which we obtain 10,000 adjusted demand $D_p$ following BPD. These numerically-obtained $D_p$ are used for a 10,000-period Monte Carlo simulation for a perishable inventory model under a base-stock policy with a base-stock level $q\in [0,50]$. We allow demands and orders for inventory replenishment to take non-integer values in our experiments. \textcolor{black}{In this numerical study, we only consider shortage and wastage costs when evaluating the expected total cost~$\mathbb{E}[C]$, omitting a holding cost from consideration (see the discussion in \S~\ref{limitation}.)}

\subsection{\textcolor{black}{Impact of $c_v$ of non-opaque products on variance $\sigma_{n,p}^2$}}
\label{5-2}
Figure~\ref{fig:cv} examines the impact of $c_v$ of the original demand $D_0$ on $\sigma_{n,p}^2$ when there are 2 non-opaque products~${(n=2)}$. We observe in Figure~\ref{fig:p-cv-b} the convergence of adjusted demand variance to the sample mean variance for various $c_v$ values; however, the comparison is not easy due to different variance of original demands with various~$c_v$ values. Thus we plot relative variance $\sigma_{rel}^2$ in Figure~\ref{fig:p-cv-b2}. Our simulation result indicates that the speed of convergence is strongly affected by $c_v$ of $D_0$; specifically, as $c_v$ gets larger, $\sigma_{rel}^2$ converges to 0 more slowly. This property is expected since we need a higher opaque proportion $p$ to suppress the demand variability of $D_0$ when $D_0$ is highly variable.
\begin{figure}[H]
    \caption{The effect of $c_v$ (or $\lambda$) on $\sigma_{n,p}^2$ and $\sigma_{rel}^2$ for $n=2$.}
    \label{fig:cv}
     \begin{subfigure}[]{0.49\textwidth}
        \centering
        \includegraphics[scale=0.45]{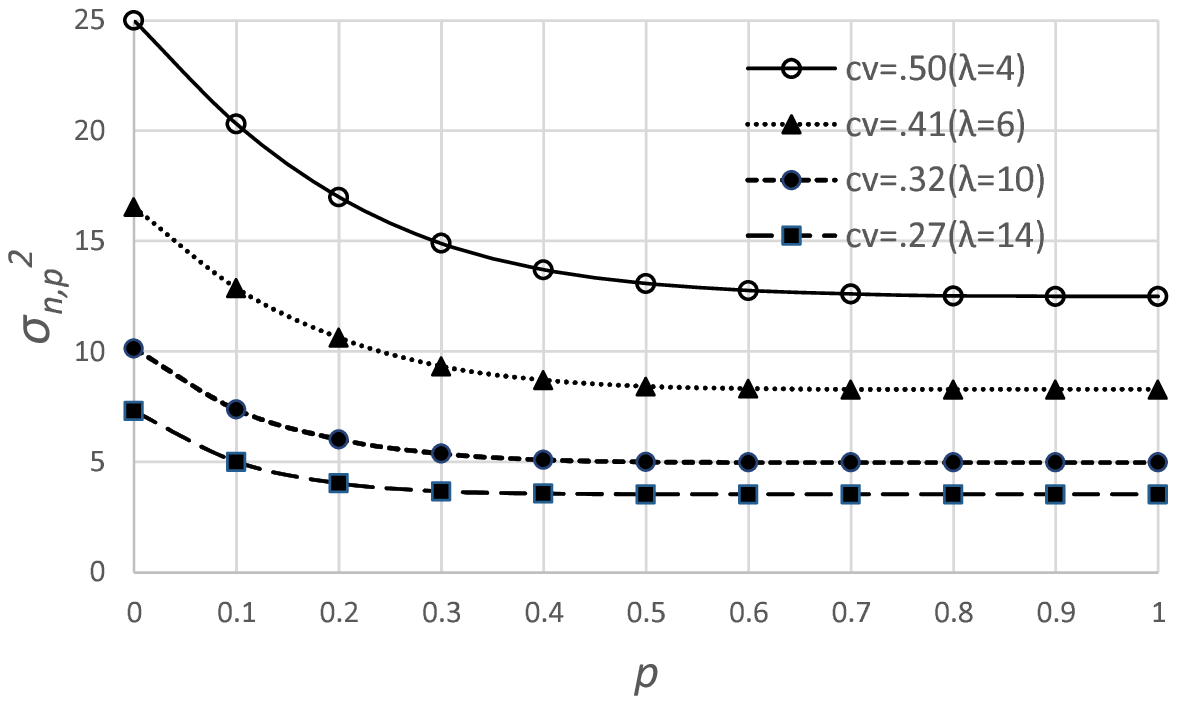}
        \caption{$\sigma_{n,p}^2$ vs $p$}
        \label{fig:p-cv-b}
    \end{subfigure}
    \begin{subfigure}[]{0.48\textwidth}
         \centering
        \includegraphics[scale=0.44]{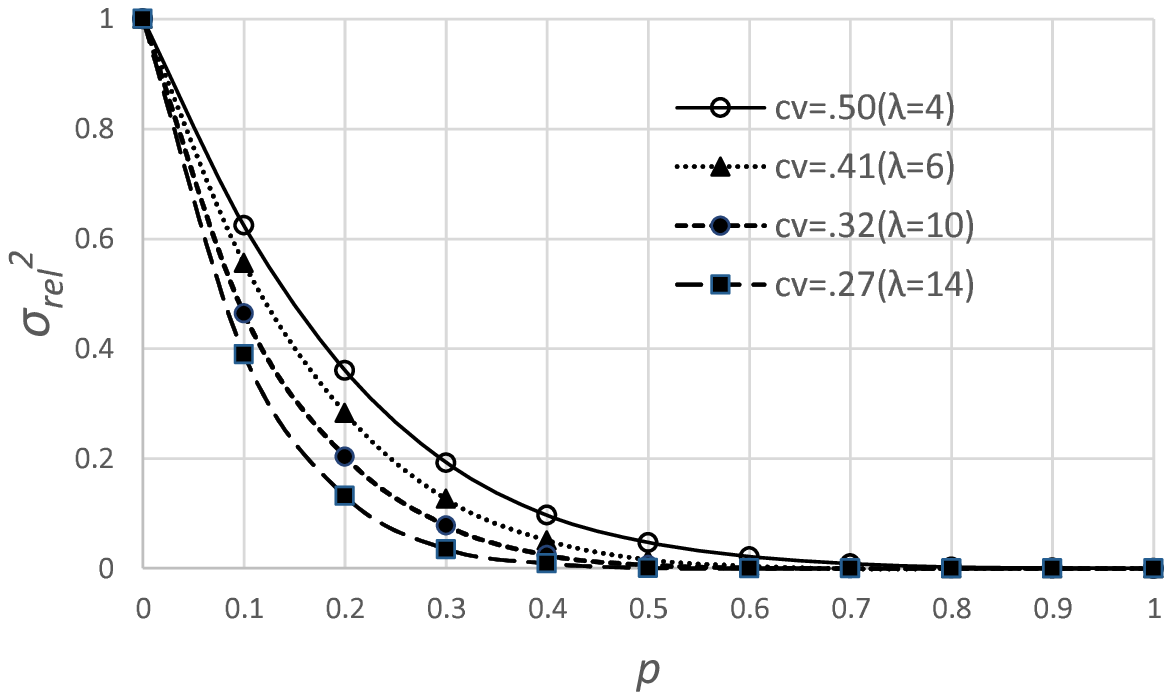}
        \caption{$\sigma_{rel}^2$ vs $p$}
        \label{fig:p-cv-b2}
   \end{subfigure}
    \centering
       \begin{subfigure}[]{0.49\textwidth}
         \centering
        \includegraphics[scale=0.45]{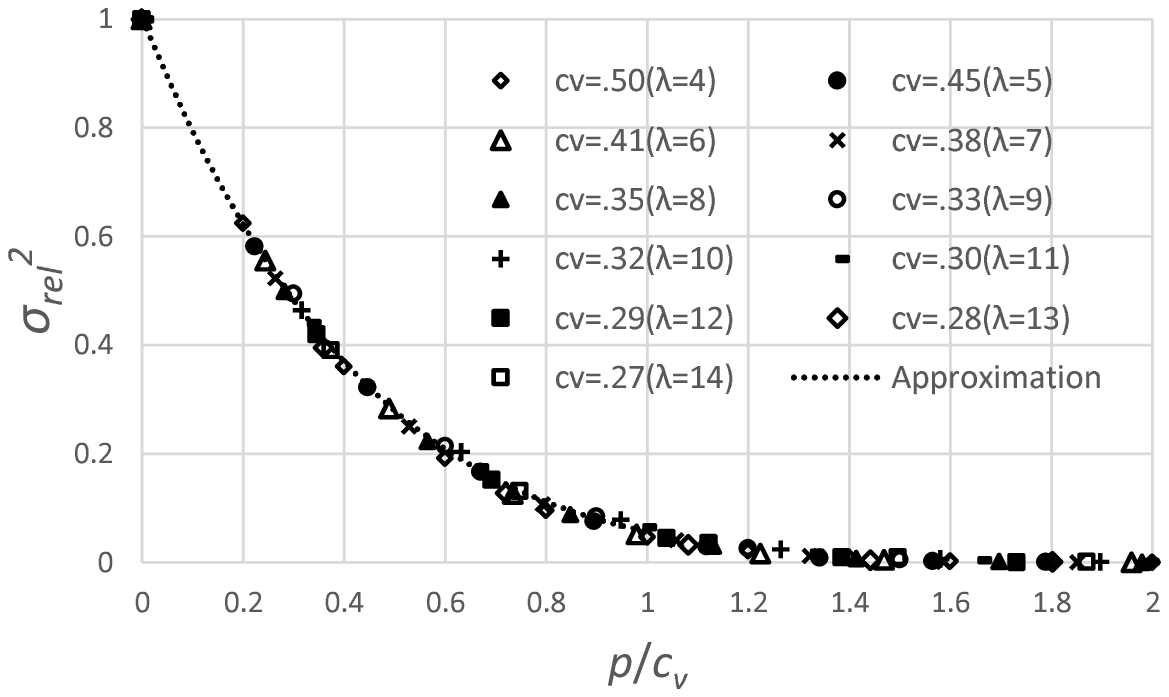}
        \caption{$\sigma_{rel}^2$ vs $p/c_v$}
        \label{fig:p-cv-a5}
   \end{subfigure}
\end{figure}
We further investigate this property in Figure~\ref{fig:p-cv-a5}. Here we plot $\sigma_{rel}^2$ as a function of $p/c_v$; the simulation results for various different parameters all align to the same curve: Equation~\eqref{sigma-rel}. This property is useful since Equation~\eqref{sigma-rel} (or Figure~\ref{fig:p-cv-a5}) determines $\sigma_{rel}^2$ for any $c_v$. For example, Figure~\ref{fig:p-cv-a5} indicates $\sigma_{rel}^2=0.2$ for $p/c_v=0.6$. Thus, if $c_v=0.5\; (\lambda=4)$, then $p=0.6 \cdot 0.5=0.3$ is required to achieve $\sigma_{rel}^2=0.2$ (or equivalently, $\sigma_{n,p}^2=12.5+0.2\cdot 12.5=15$ following Equation~\eqref{eq:relativevariance}), which is consistent with the result in Figure~\ref{fig:p-cv-b2}. We discuss $\sigma_{rel}^2$ only for the $n=2$ case, but the following result implies that $\sigma_{rel}^2$ is insensitive to $n$. 

\subsection{\textcolor{black}{Impact of the number of non-opaque products $n$ on variance $\sigma_{n,p}^2$}}
\label{5-3}

Figure~\ref{fig:p-npr-b2} examines the impact of the number of non-opaque products $n\; (\ge 2)$ on $\sigma_{n,p}^2$. To examine the speed of convergence ($\sigma_{n,p}^2 \rightarrow \sigma_{n,1}^2=\sigma^2/n$ as $p \rightarrow 1$), we fix ${c_v=0.32 \:(\lambda=10)}$ and plot $\sigma_{rel}^2$ for each $n$ in Figure~\ref{fig:p-npr-a2}. Figure~\ref{fig:p-npr-a2} shows an approximate insensitivity property: $\sigma_{rel}^2$ is \emph{almost} independent from $n$. Thus, for any $n$ or $c_v$, we can derive an approximate value of $\sigma_{n,p}^2$ using Equations~\eqref{eq:relativevariance} and \eqref{sigma-rel}. For example, Figure~\ref{fig:p-cv-a5} indicates that $\sigma_{rel}^2=0.2$ at $p/c_v=0.6$, thus, for $n=4$ and $c_v=0.32\: (\lambda=10)$, we obtain $\sigma_{n,p}^2=2.5+0.2\cdot 7.5=4$ at $p=0.6 \cdot 0.32=0.2$; this result is consistent with Figure~\ref{fig:p-npr-b2}.
\begin{figure}[H]
    \caption{The effect of $n$ on $\sigma_{n,p}^2$ and $\sigma_{rel}^2$ for $c_v=0.32 \;(\lambda=10)$.}
    \label{fig:npr}
    \centering
   \begin{subfigure}[]{0.49\textwidth}
        \centering
        \includegraphics[scale=0.45]{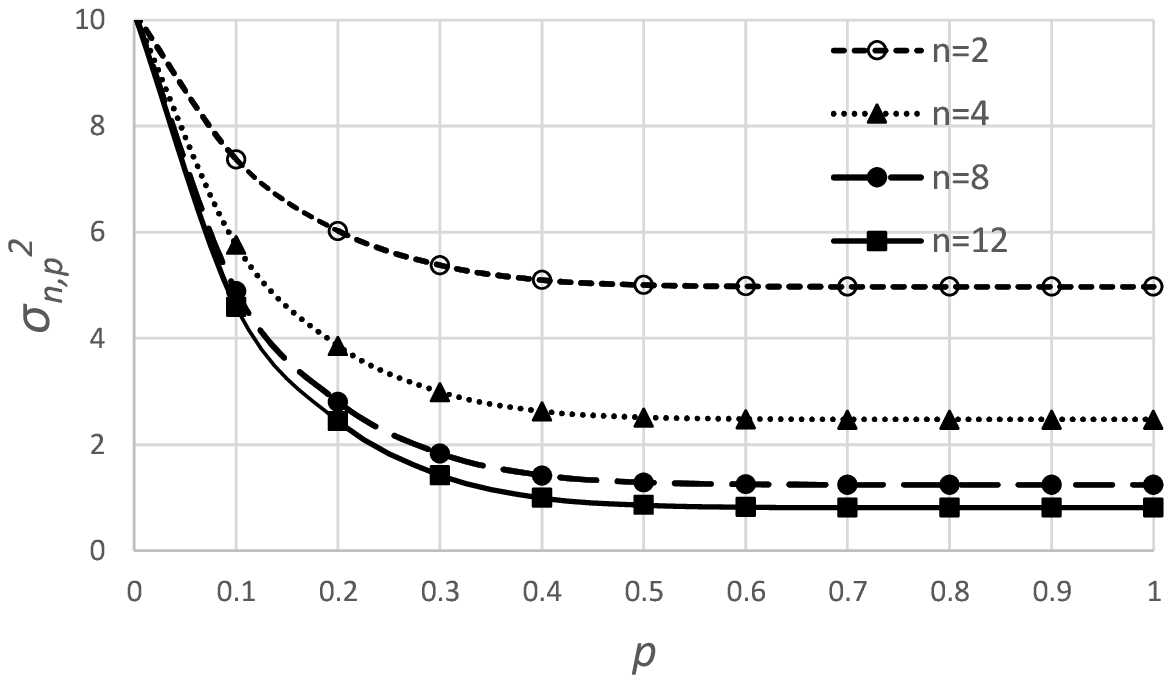}
        \caption{$\sigma_{n,p}^2$ vs $p$}
        \label{fig:p-npr-b2}
   \end{subfigure}
          \begin{subfigure}[]{0.49\textwidth}
     \centering
        \includegraphics[scale=0.45]{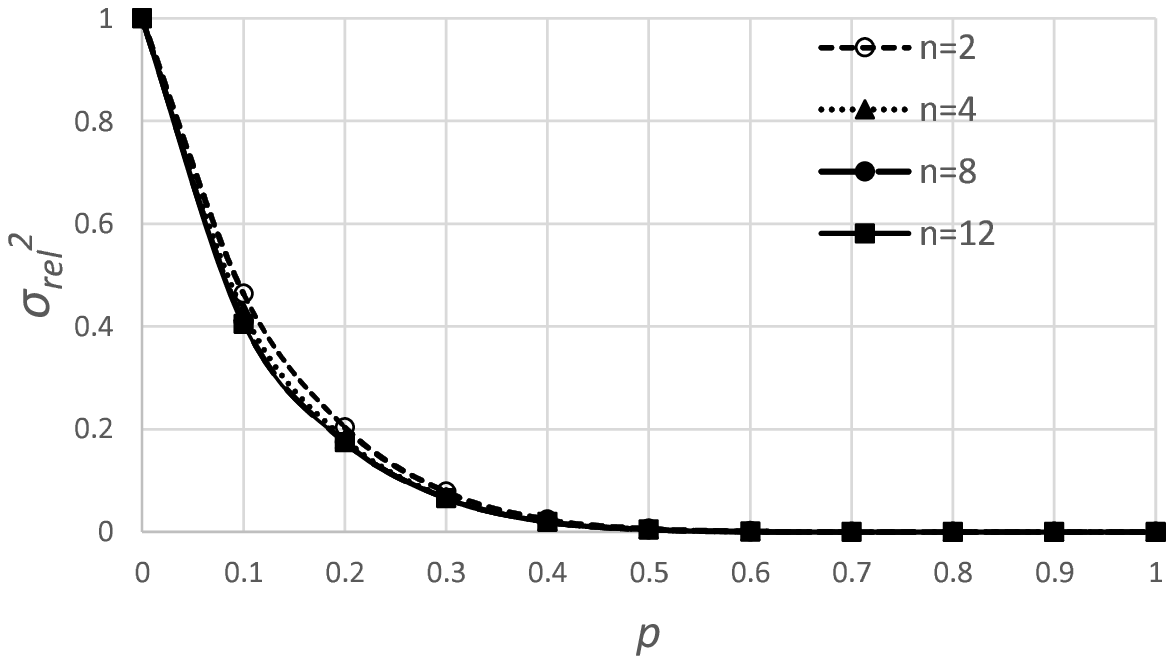}
        \caption{$\sigma_{rel}^2$ vs $p$}
        \label{fig:p-npr-a2}
    \end{subfigure}
\end{figure}

\subsection{Impact of the $(n,p)$ opaque scheme on the expected total cost $\mathbb{E}[C]$}
\label{5-4}
In the previous section, we observed that $\sigma_{n,p}^2$ of adjusted demand $D_p$ can be arbitrarily reduced by controlling $(n,p)$ under the opaque scheme. We now examine how $D_p$ and its variance $\sigma_{n,p}^2$ under the $(n,p)$ opaque scheme affect the expected per-period total cost $\mathbb{E}[C]=r\mathbb{E}[S]+\theta \mathbb{E}[W]$. We also find the condition to make $\mathbb{E}[C] \approx 0$.

Figure~\ref{fig:shortageholdingwastage} shows the impact of $p$ on $\mathbb{E}[S]$ and $\mathbb{E}[W]$ while fixing $c_v=0.32$ ($\lambda=10$) and $n=12$; $\mathbb{E}[S]$ and $\mathbb{E}[W]$ are then combined to obtain $\mathbb{E}[C]$ for $m=2$ in Figure~\ref{fig:m2ptotal} and for $m=3$ in Figure~\ref{fig:m3ptotal}. \textcolor{black}{We set $r=\theta=1$ to draw Figures~\ref{fig:m2ptotal} and~\ref{fig:m3ptotal}. (Note: only the ratio $r/\theta$ is important since $\mathbb{E}[C]$ has an arbitrary unit in our plots.)} We observe that $\mathbb{E}[C] \approx 0$ for a wider range of $q$ when the proportion $p$ is higher. Note that $\mathbb{E}[S]$ for $m=2$ and $m=3$ (Figures \ref{fig:m2shortage} and \ref{fig:m3shortage}, respectively) are identical since shortages do not depend on shelflife $m$, while $\mathbb{E}[W]$ (Figures \ref{fig:m2wastage} and \ref{fig:m3wastage}) show strong dependency on $m$; larger $m$ reduces $\mathbb{E}[W]$ given $q$. The optimal base-stock level, which is determined by the trade-off between shortage and wastage, thus depends on $m$ (as well as cost parameters, which we assume $r=\theta=1$); in this experiment we observe $q=15$ for $m=2$ (Figure~\ref{fig:m2ptotal}) and $q=18$ for $m=3$ (Figure~\ref{fig:m3ptotal}). Since $\mathbb{E}[W]$ is smaller for larger $m$, the optimal $\mathbb{E}[C]$ is smaller and the optimal base-stock level that minimizes $\mathbb{E}[C]$ is larger for larger $m$.
\begin{figure}[H]
    \caption{The effect of $p$ on shortages, wastages, and total cost for $c_v=0.32 \;(\lambda=10), n=12$.}
    \label{fig:shortageholdingwastage}
    \centering
    \begin{subfigure}[]{0.48\textwidth}
     \centering
        \includegraphics[scale=0.44]{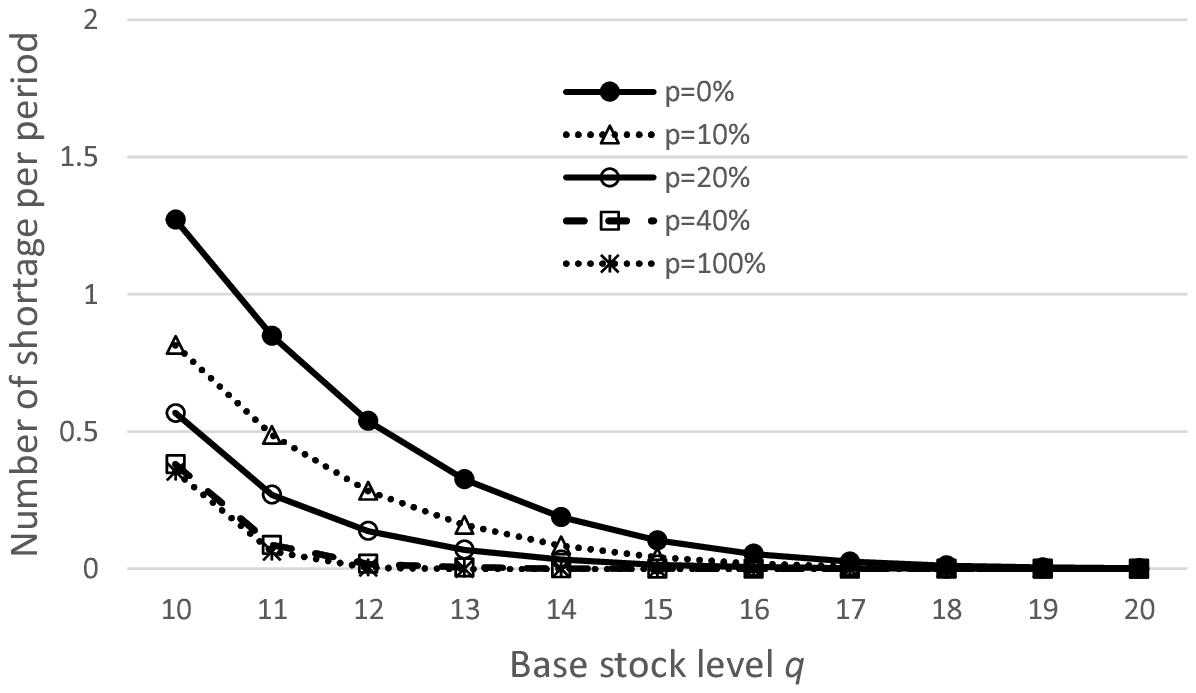}
        \caption{Shortages; $q\in [10,20]$}
        \label{fig:m2shortage}
   \end{subfigure}
       \centering
    \begin{subfigure}[]{.48\textwidth}
     \centering
        \includegraphics[scale=0.44]{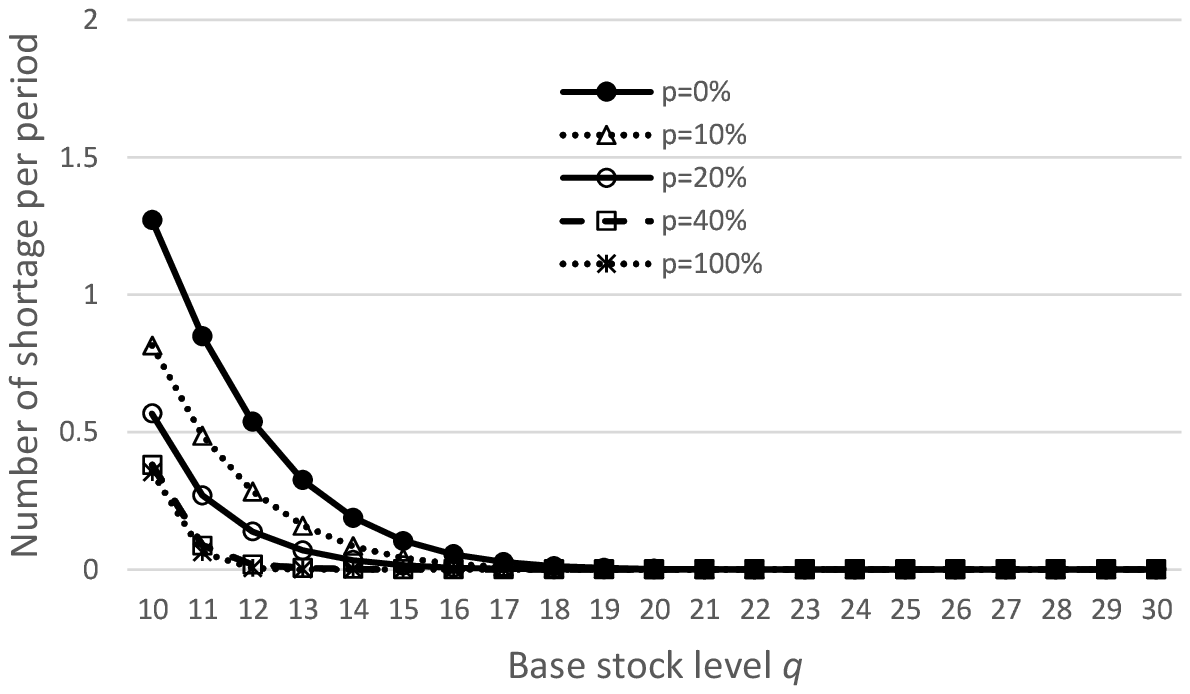}
        \caption{Shortages; $q\in [10,30]$}
        \label{fig:m3shortage}
   \end{subfigure}   
    \begin{subfigure}[]{0.48\textwidth}
     \centering
        \includegraphics[scale=0.44]{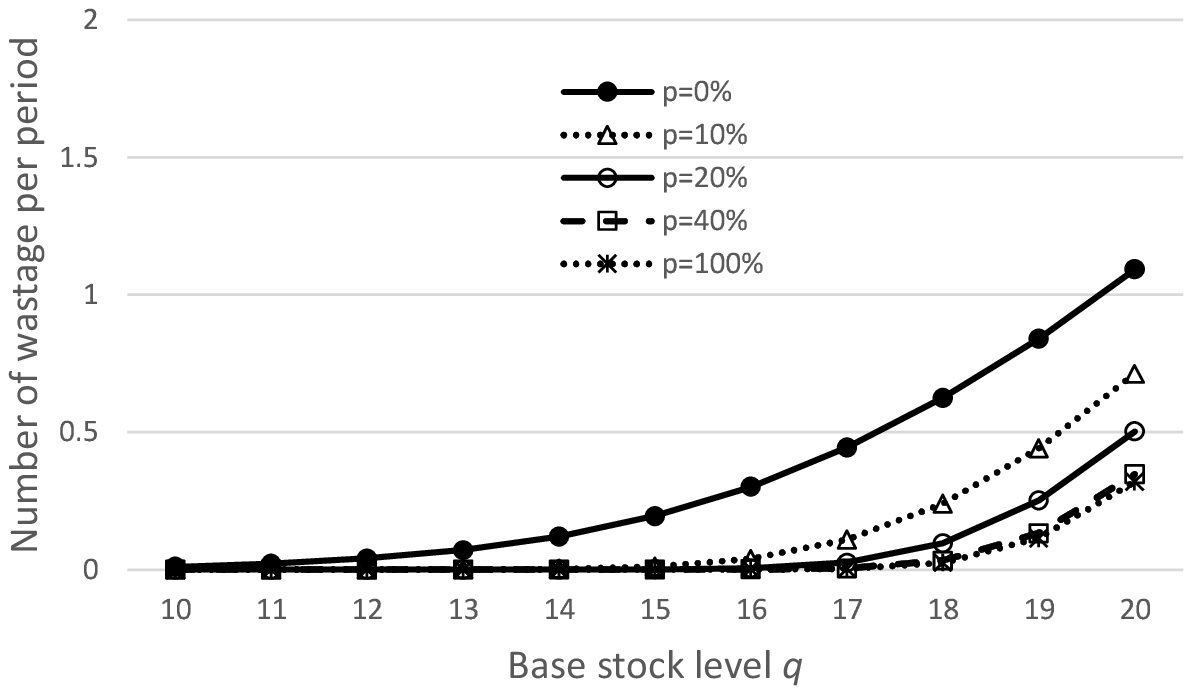}
        \caption{Wastages ($m=2$)}
        \label{fig:m2wastage}
   \end{subfigure}
     \begin{subfigure}[]{0.48\textwidth}
     \centering
        \includegraphics[scale=0.44]{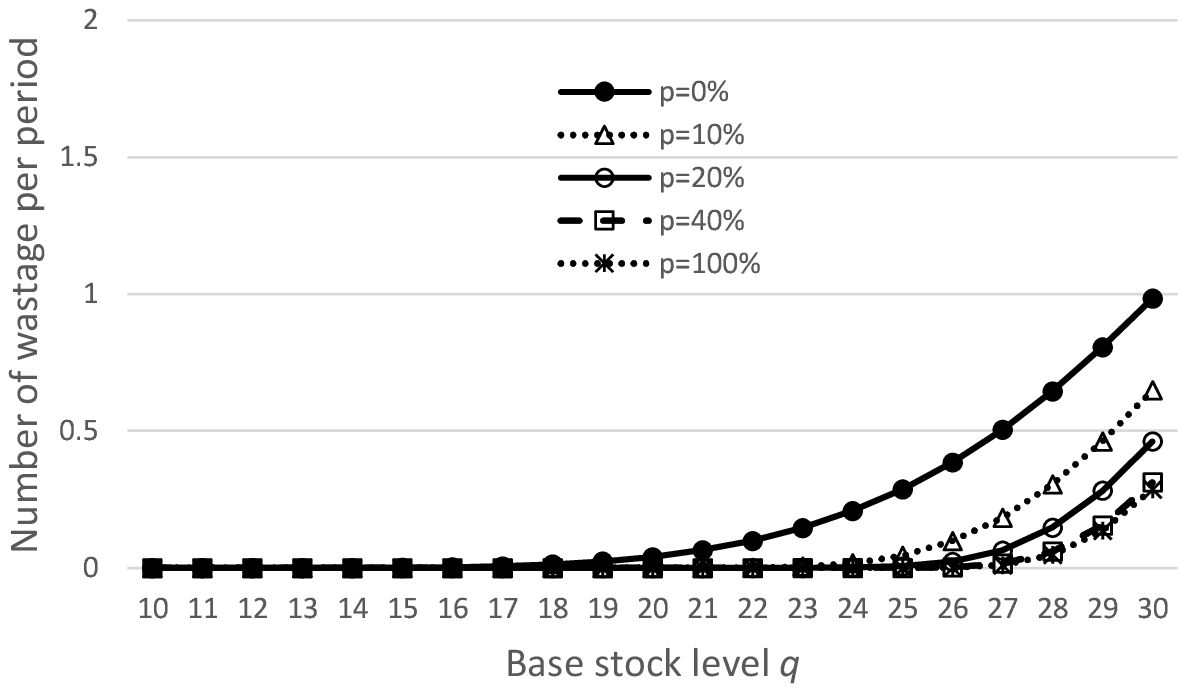}
        \caption{Wastages ($m=3$)}
        \label{fig:m3wastage}
    \end{subfigure}
        \begin{subfigure}[]{0.48\textwidth}
     \centering
        \includegraphics[scale=0.44]{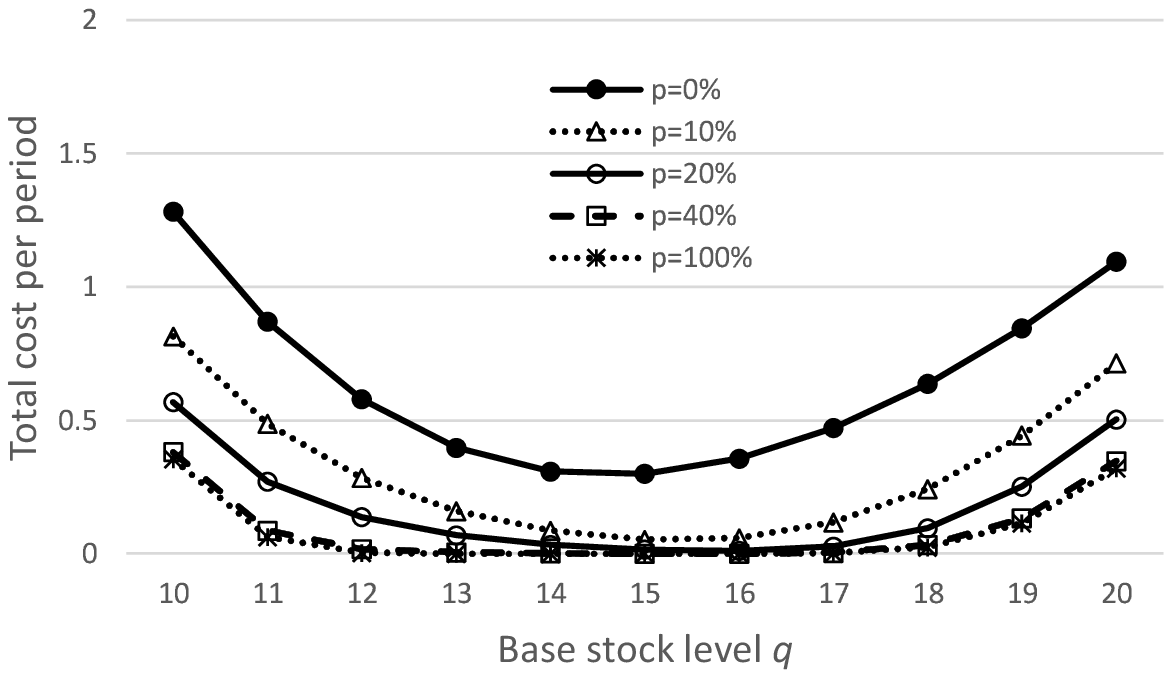}
        \caption{Total Cost ($m=2,r=1,\theta=1$)}
        \label{fig:m2ptotal}
   \end{subfigure}
     \begin{subfigure}[]{0.48\textwidth}
     \centering
        \includegraphics[scale=0.44]{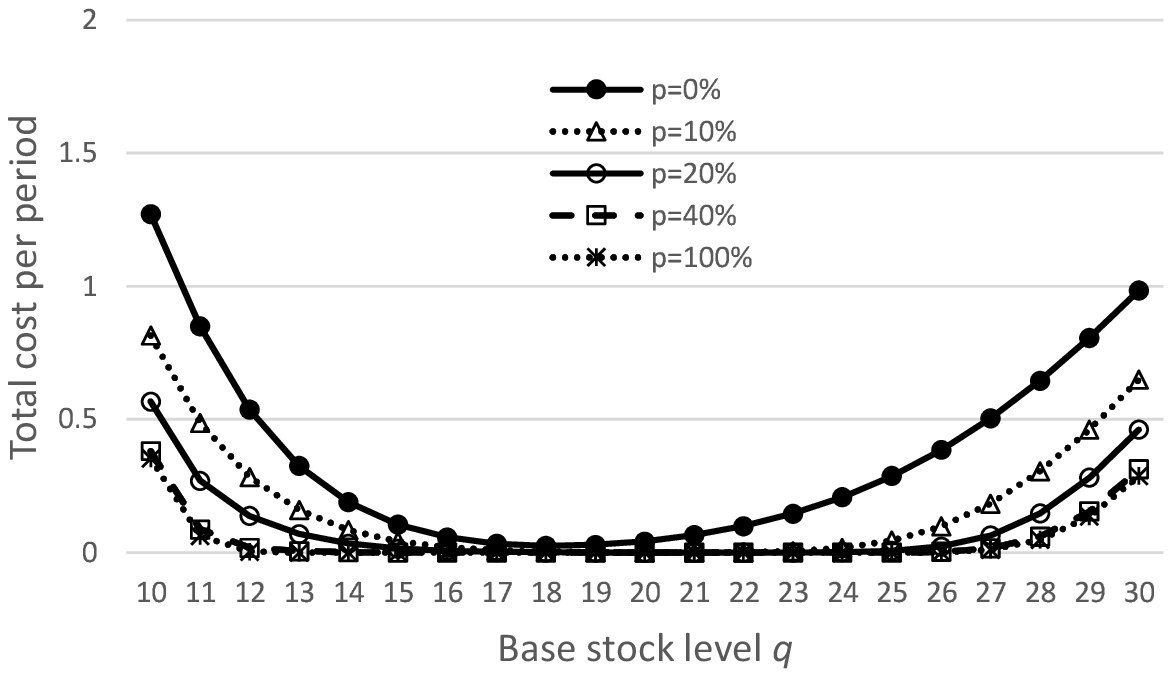}
        \caption{Total Cost ($m=3,r=1,\theta=1$)}
        \label{fig:m3ptotal}
    \end{subfigure}
\begin{flushleft}
\hspace{-0.2cm} \footnotesize \emph{Notes}: Plots in (a) and (b) are equivalent except for the range; we use $[10,20]$ for $m=2$ while we use $[10,30]$ for $m=3$.  Both $p=0\%$ lines in (e) and (f) correspond to the baseline case with no opaque scheme.
\end{flushleft}
\end{figure}

\textcolor{black}{Next, we consider some cases with~${r \neq \theta}$. We set the ratio between shortage and wastage costs to $r/\theta=2$ (Figures~\ref{fig:m2ptotalr2} and~\ref{fig:m3ptotalr2}) and $r/\theta=0.5$ (Figures~\ref{fig:m2ptotalth2} and~\ref{fig:m3ptotalth2}) and observe the plots of $\mathbb{E}[C]$. As expected, when $r/\theta>1$ ($r/\theta<1$), shortages (wastages, respectively) play more important role than $r/\theta=1$ case. We continue to observe the benefit of using the opaque scheme: As $p$ increases, $\mathbb{E}[C]$ goes down. Because the property of $\mathbb{E}[C]$ does not change qualitatively (although values change), we use $r=\theta=1$ in subsequent numerical experiments.}
\begin{figure}[H]
    \caption{The effect of $r/\theta$ on total cost for $c_v=0.32 \;(\lambda=10), n=12$.}
    \label{fig:randtheta}
    \centering
     \begin{subfigure}[]{0.48\textwidth}
     \centering
        \includegraphics[scale=0.44]{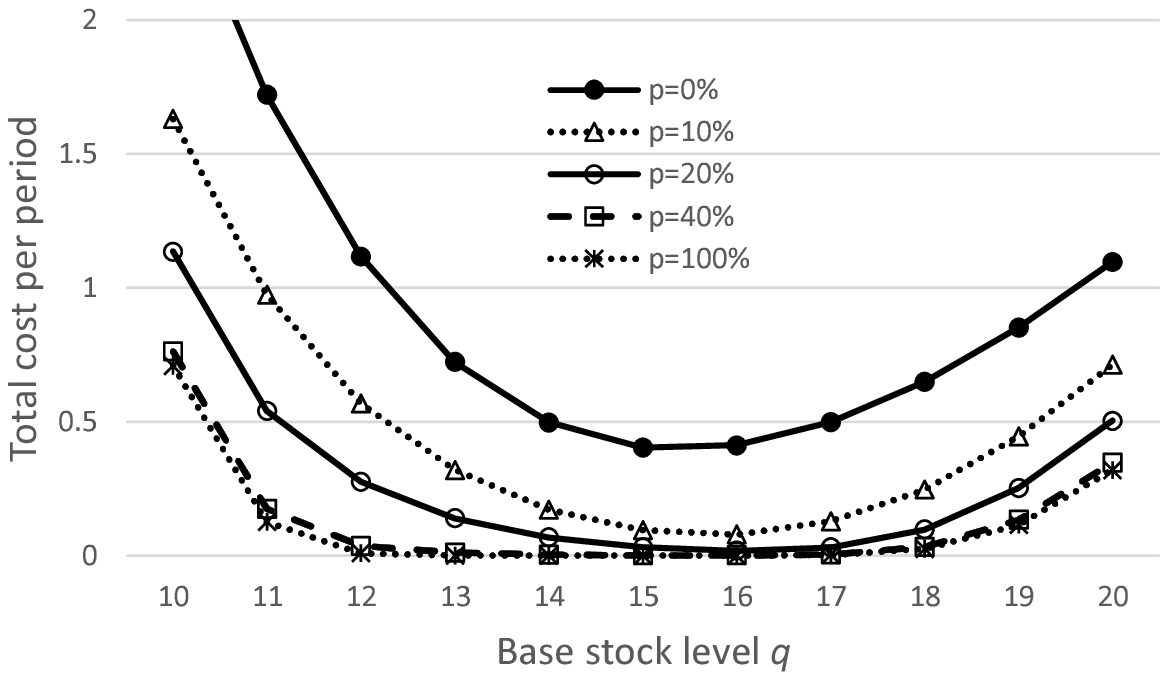}
        \caption{Total Cost ($m=2,r=2,\theta=1$)}
        \label{fig:m2ptotalr2}
   \end{subfigure}
     \begin{subfigure}[]{0.48\textwidth}
     \centering
        \includegraphics[scale=0.44]{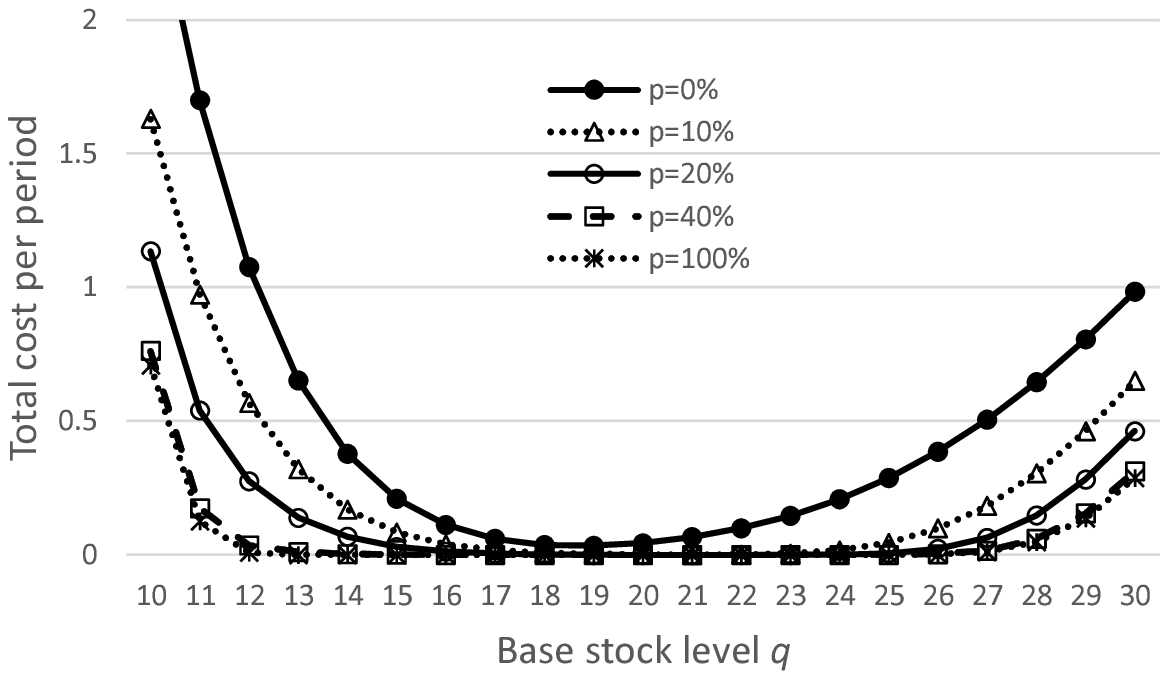}
        \caption{Total Cost ($m=3,r=2,\theta=1$)}
        \label{fig:m3ptotalr2}
    \end{subfigure}
         \begin{subfigure}[]{0.48\textwidth}
     \centering
        \includegraphics[scale=0.44]{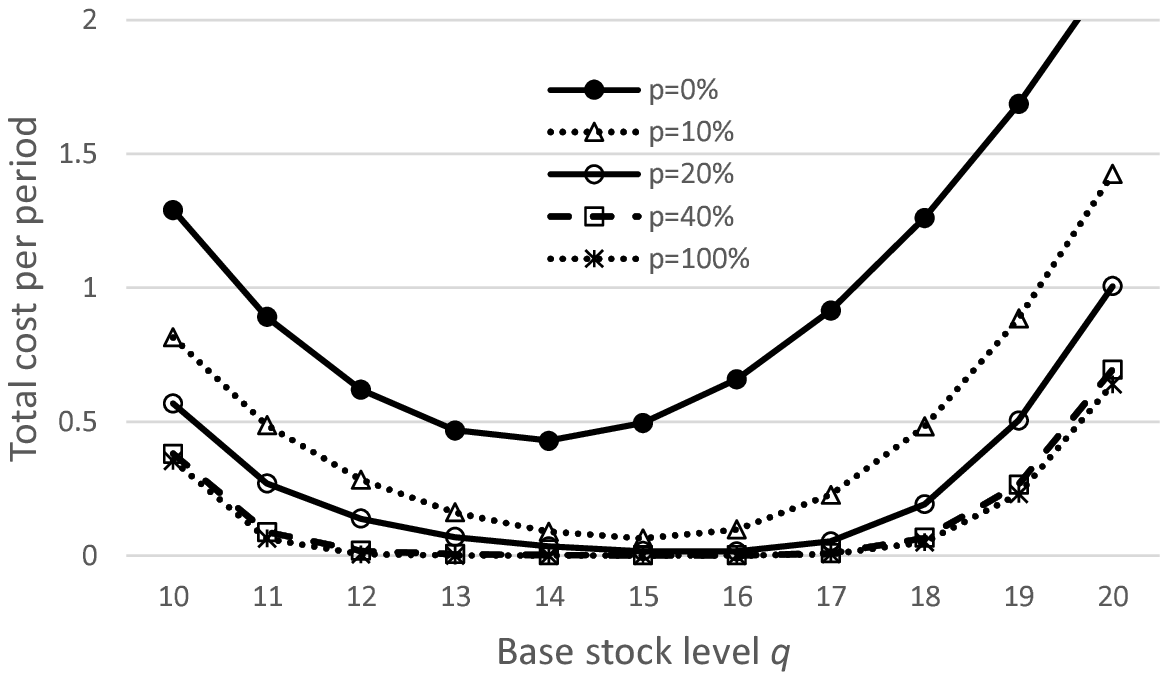}
        \caption{Total Cost ($m=2,r=1,\theta=2$)}
        \label{fig:m2ptotalth2}
   \end{subfigure}
     \begin{subfigure}[]{0.48\textwidth}
     \centering
        \includegraphics[scale=0.44]{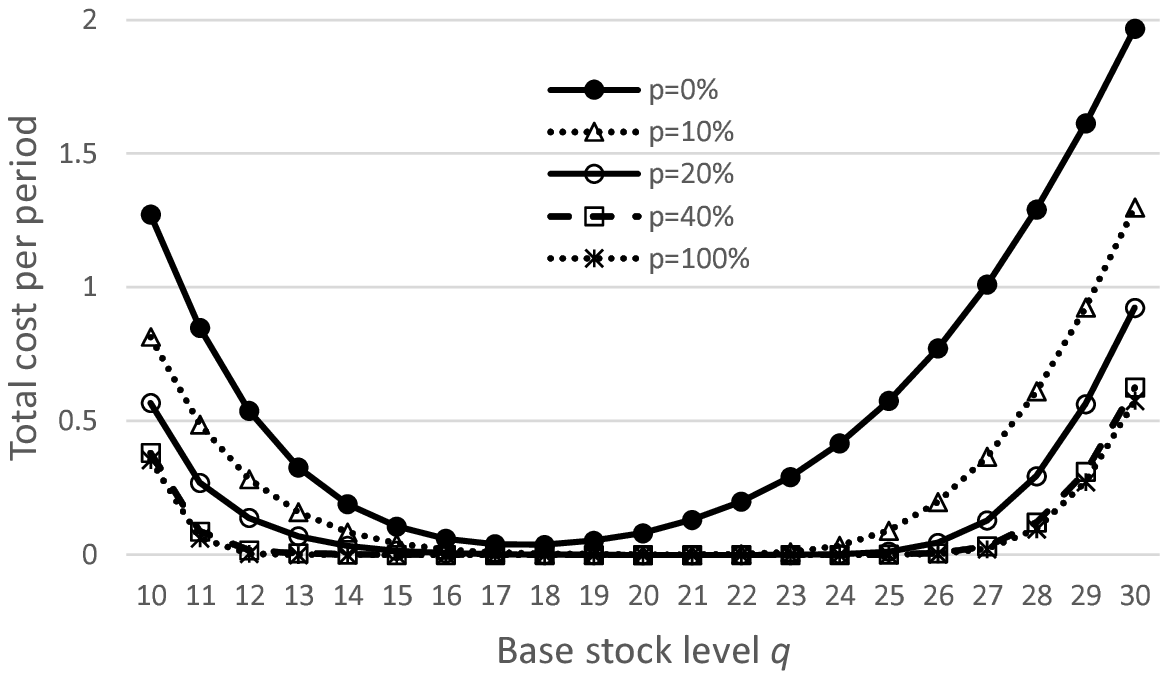}
        \caption{Total Cost ($m=3,r=1,\theta=2$)}
        \label{fig:m3ptotalth2}
    \end{subfigure}
\end{figure}

Following the same procedure as above, we vary $n$ while fixing $c_v=0.32$ ($\lambda=10$) and $p=1$, and obtain $\mathbb{E}[C]$ for $m=2$ in Figure~\ref{fig:m2ntotal} and for $m=3$ in Figure~\ref{fig:m3ntotal}. We observe similarities between Figures~\ref{fig:m2ptotal} and \ref{fig:m2ntotal} and between Figures~\ref{fig:m3ptotal} and \ref{fig:m3ntotal}. For example, $n=12, p=10\%$ lines in Figures~\ref{fig:m2ptotal} and \ref{fig:m3ptotal} closely resemble $n=2, p=100\%$ lines in Figures~\ref{fig:m2ntotal} and \ref{fig:m3ntotal}, respectively; $n=12, p=20\%$ and $n=4, p=100\%$ lines also look very similar to each other. These similarities are due to the proximity of variances: $\sigma_{12,0.1}^2\approx \sigma_{2,1}^2=\sigma^2/2$ and $\sigma_{12,0.2}^2\approx \sigma_{4,1}^2=\sigma^2/4$ hold according to Figure~\ref{fig:p-cv-b}. This observation implies that the value of $\sigma_{n,p}^2$ is a major determinant factor for $\mathbb{E}[C]$ under the opaque scheme. (Note: we observe this property more explicitly in Figure~\ref{fig:vartotalcost}.)
\begin{figure}[H]
    \caption{The effect of $n$ on total cost for $c_v=0.32 \;(\lambda=10), p=1$.}
    \label{fig:totalcost}
    \centering
        \begin{subfigure}[]{0.48\textwidth}
     \centering
        \includegraphics[scale=0.44]{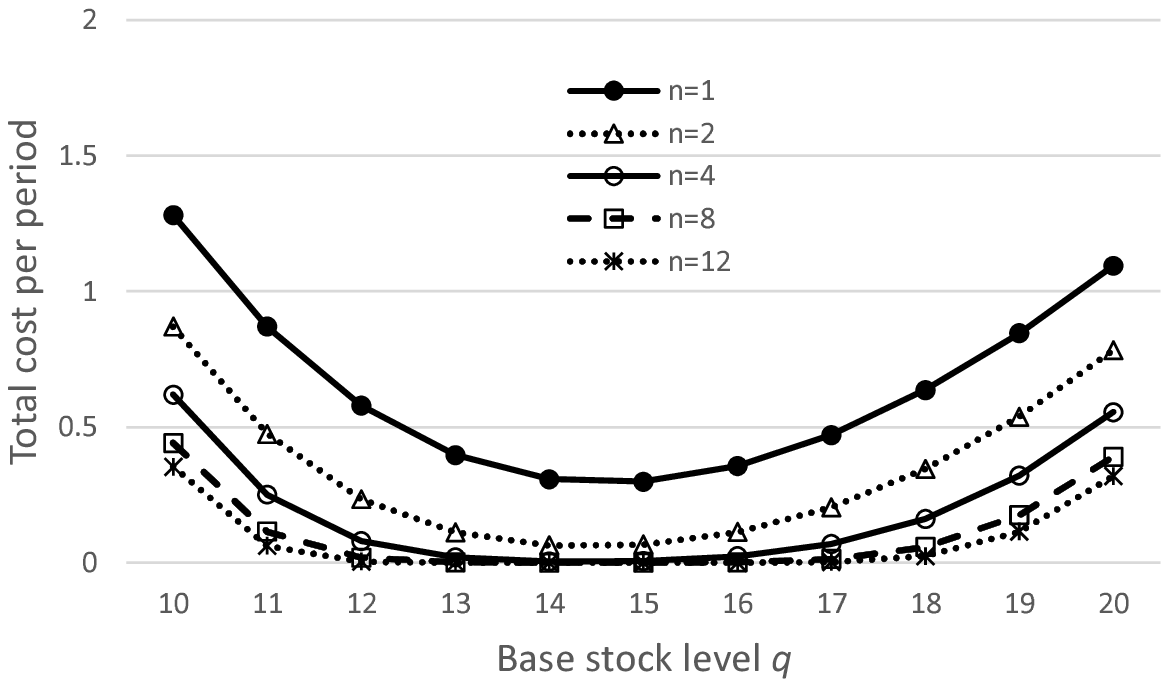}
        \caption{Total Cost ($m=2$)}
        \label{fig:m2ntotal}
   \end{subfigure}
     \begin{subfigure}[]{0.48\textwidth}
     \centering
        \includegraphics[scale=0.44]{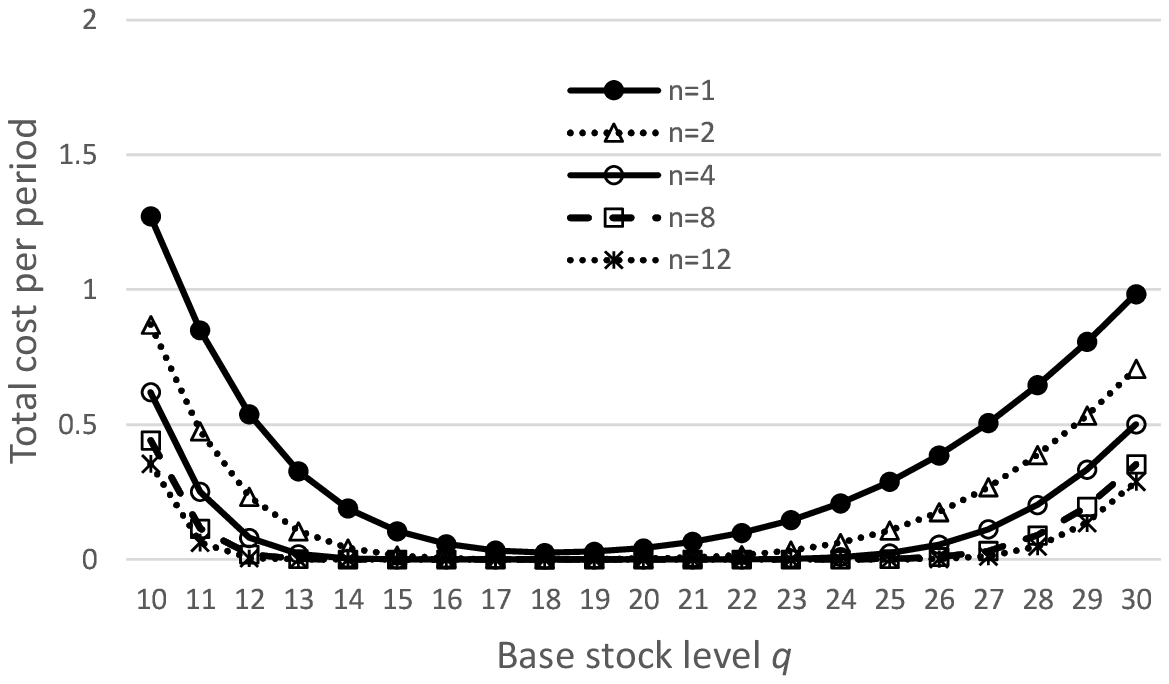}
        \caption{Total Cost ($m=3$)}
        \label{fig:m3ntotal}
    \end{subfigure}
\begin{flushleft}
\hspace{-.1cm} \footnotesize \emph{Notes}: This figure and all other figures below assume $r=\theta=1.$ Both $n=1$ lines in (a) and (b) correspond to the baseline case with no opaque scheme.
\end{flushleft}
\end{figure}

We observe a diminishing return property for the $(n,p)$ opaque scheme in Figures~\ref{fig:shortageholdingwastage} and \ref{fig:totalcost}: the marginal impact on $\mathbb{E}[C]$ is larger when $p$ or $n$ are smaller. For example, a change from $p=0\%$ to $10\%$ in Figure \ref{fig:m2ptotal} and a change from $n=1$ to 2 in Figure \ref{fig:m2ntotal} reduce $\mathbb{E}[C]$ most (compared to other changes such as a change from $p=10\%$ to $20\%$). This is also attributed to the dependence of $\mathbb{E}[C]$ on $\sigma_{n,p}^2$: We observe a larger marginal reduction of $\sigma_{n,p}^2$ when $p$ or $n$ are smaller according to Figure~\ref{fig:npr}.


\subsection{\textcolor{black}{Impact of $c_v$ of non-opaque products on the expected total cost $\mathbb{E}[C]$}}
\label{5-5}

\begin{figure}[H]
    \caption{Effect of $c_v$ on $\mathbb{E}[C]$ for $n=2$.}
    \label{fig:lam}
    \centering
    \begin{subfigure}[]{0.48\textwidth}
     \centering
        \includegraphics[scale=0.44]{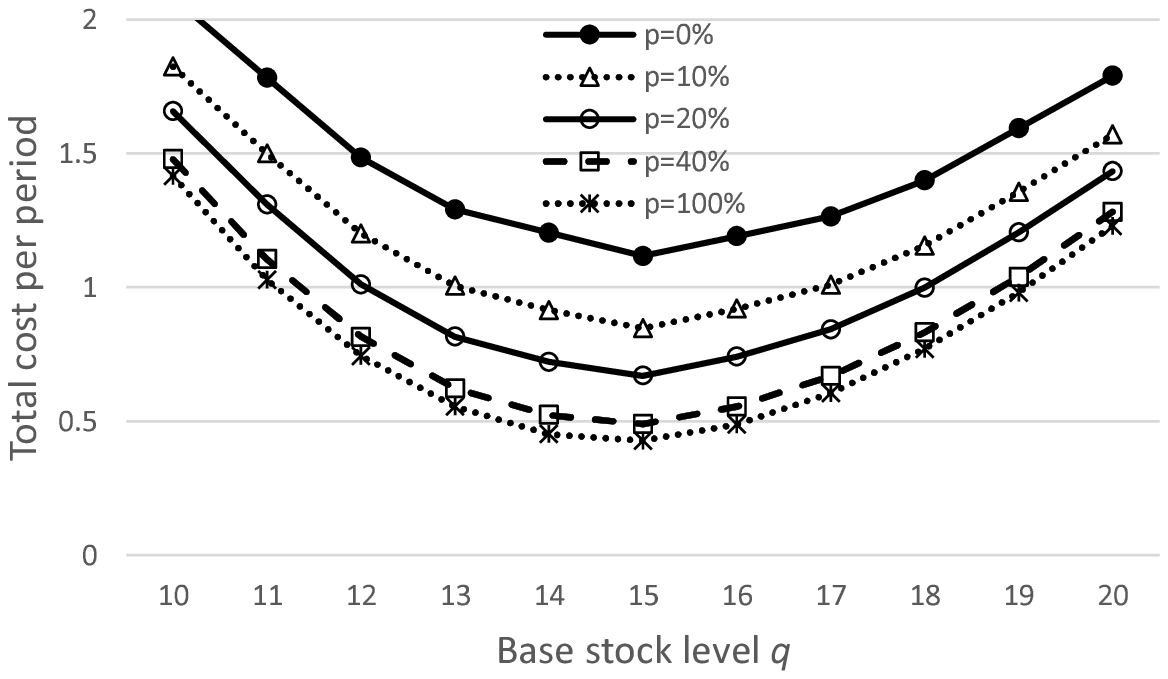}
        \caption{$m=2, c_v=0.50 \;(\lambda=4)$ case}
        \label{fig:lam4-a}
   \end{subfigure}
     \begin{subfigure}[]{0.48\textwidth}
     \centering
        \includegraphics[scale=0.44]{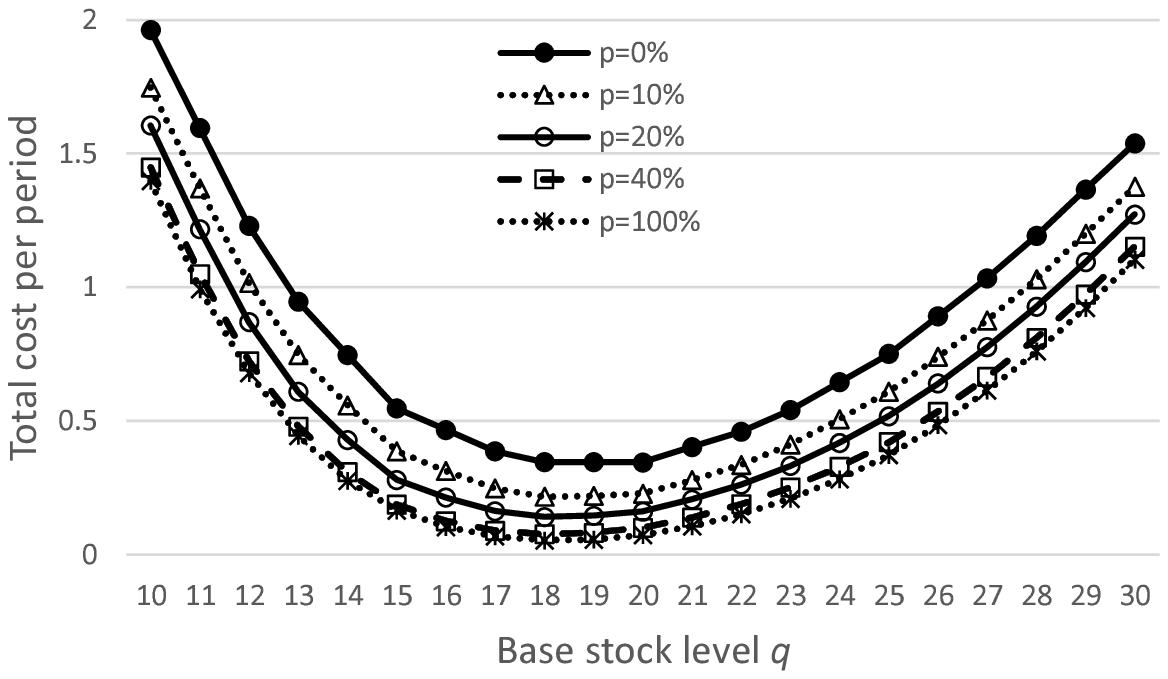}
        \caption{$m=3, c_v=0.50 \;(\lambda=4)$ case}
        \label{fig:lam4-b}
    \end{subfigure}
        \begin{subfigure}[]{0.48\textwidth}
     \centering
        \includegraphics[scale=0.44]{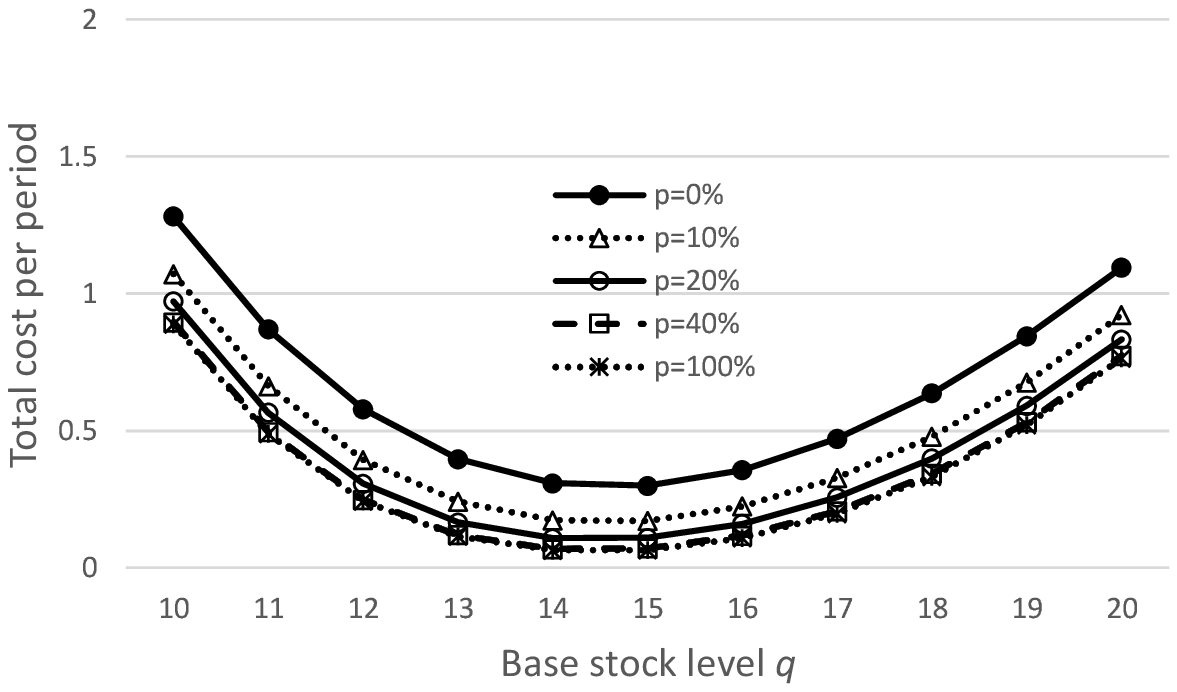}
        \caption{$m=2, c_v=0.32 \;(\lambda=10)$ case}
        \label{fig:lam10-a}
   \end{subfigure}
     \begin{subfigure}[]{0.48\textwidth}
     \centering
        \includegraphics[scale=0.44]{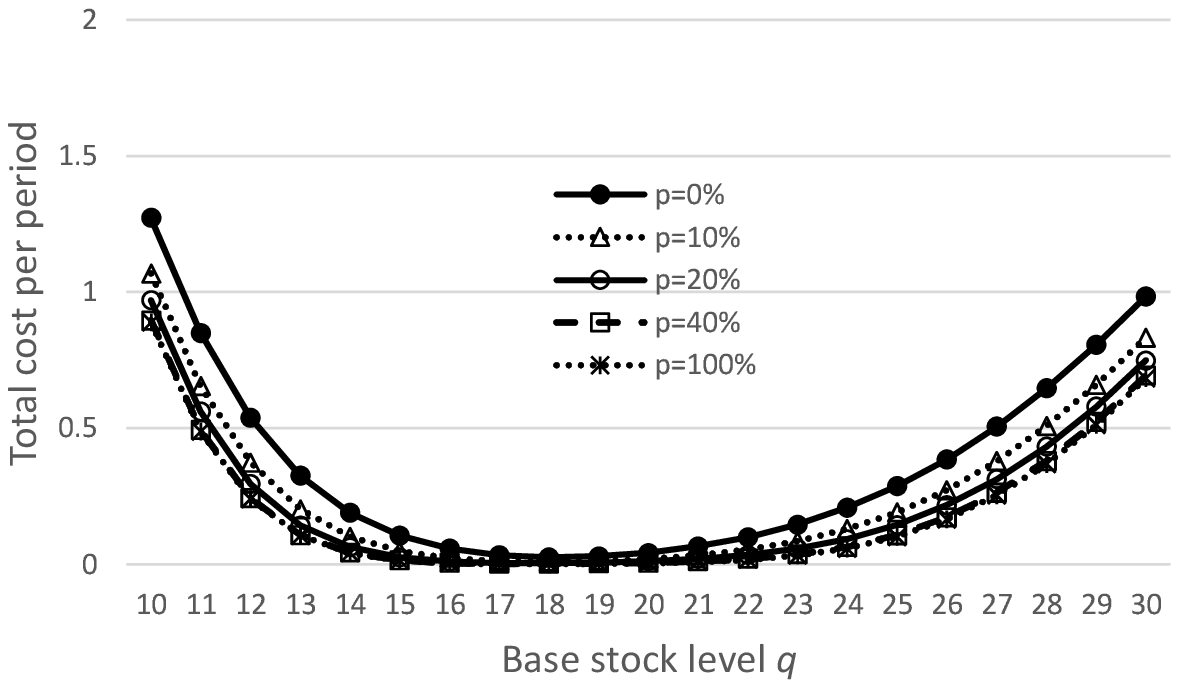}
        \caption{$m=3, c_v=0.32 \;(\lambda=10)$ case}
        \label{fig:lam10-b}
    \end{subfigure}
    \begin{subfigure}[]{0.48\textwidth}
     \centering
        \includegraphics[scale=0.44]{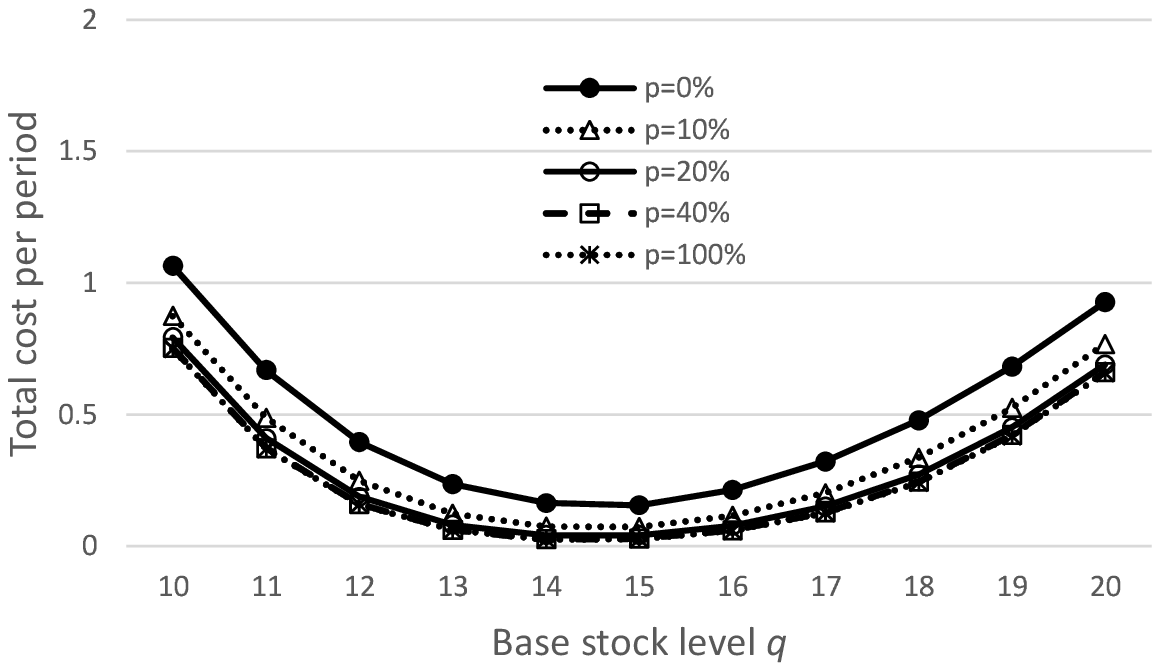}
        \caption{$m=2, c_v=0.27 \;(\lambda=14)$ case}
        \label{fig:lam14-a}
   \end{subfigure}
     \begin{subfigure}[]{0.48\textwidth}
     \centering
        \includegraphics[scale=0.44]{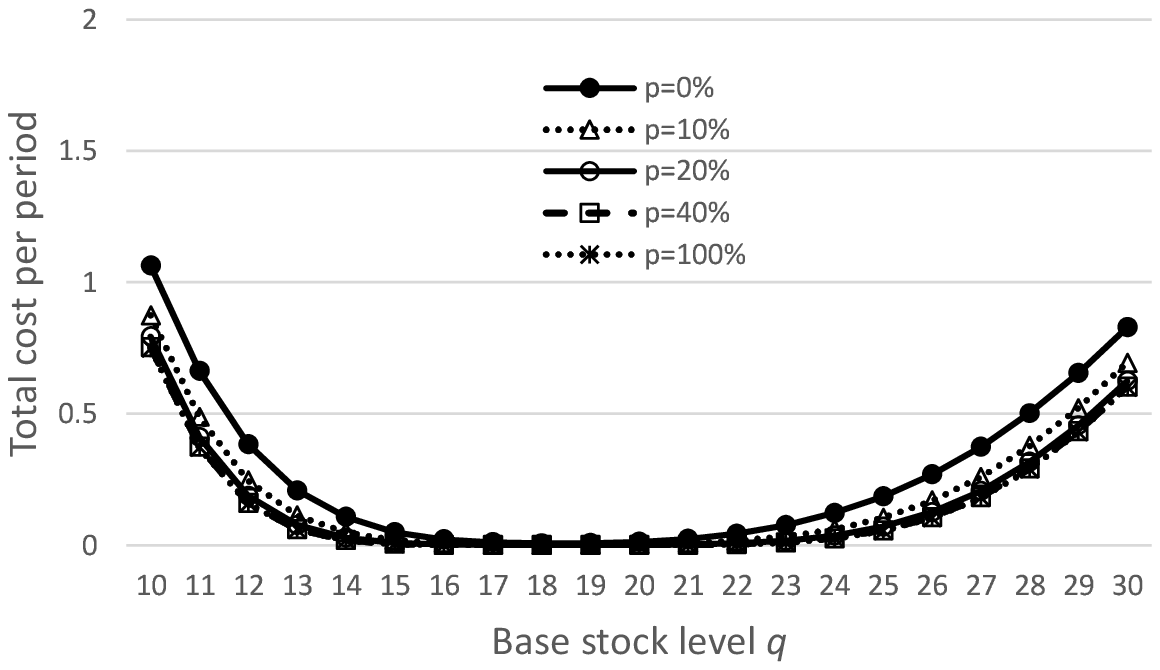}
        \caption{$m=3, c_v=0.27 \;(\lambda=14)$ case}
        \label{fig:lam14-b}
    \end{subfigure}
\end{figure}
We next examine the impact of variability ($c_v$) of original demand $D_0$ on~$\mathbb{E}[C]$. In Figure \ref{fig:lam}, we use $c_v \in \{0.50, 0.32, 0.27\}$ for $n=2$ and compare their results. First, we observe that $\mathbb{E}[C]$ is closer to zero for smaller $c_v$ given $q$; this observation is attributed to the smaller $\sigma_{n,p}^2$ for smaller $c_v$ (see also Figure~\ref{fig:p-cv-b}). Second,~$\mathbb{E}[C]$ converges faster to its minimum for smaller $c_v$ as $p \rightarrow 1$; this observation is attributed to the increased speed of convergence of ${\sigma_{n,p}^2}$ to its minimum $\sigma_{n,1}^2$ for smaller $c_v$ (see also Figure~\ref{fig:p-cv-b2}). These results imply that the~$(n,p)$ opaque scheme impacts ${\mathbb{E}[C]}$ mainly through the change in the variance~${\sigma_{n,p}^2}$ of $D_p$

\begin{figure}[H]
    \caption{Effect of $\sigma_{n,p}^2$ on $\mathbb{E}[C]$ for $c_v=0.32 \;(\lambda=10)$.}
    \label{fig:vartotalcost}
    \centering
    \begin{subfigure}[]{0.48\textwidth}
     \centering
        \includegraphics[scale=0.44]{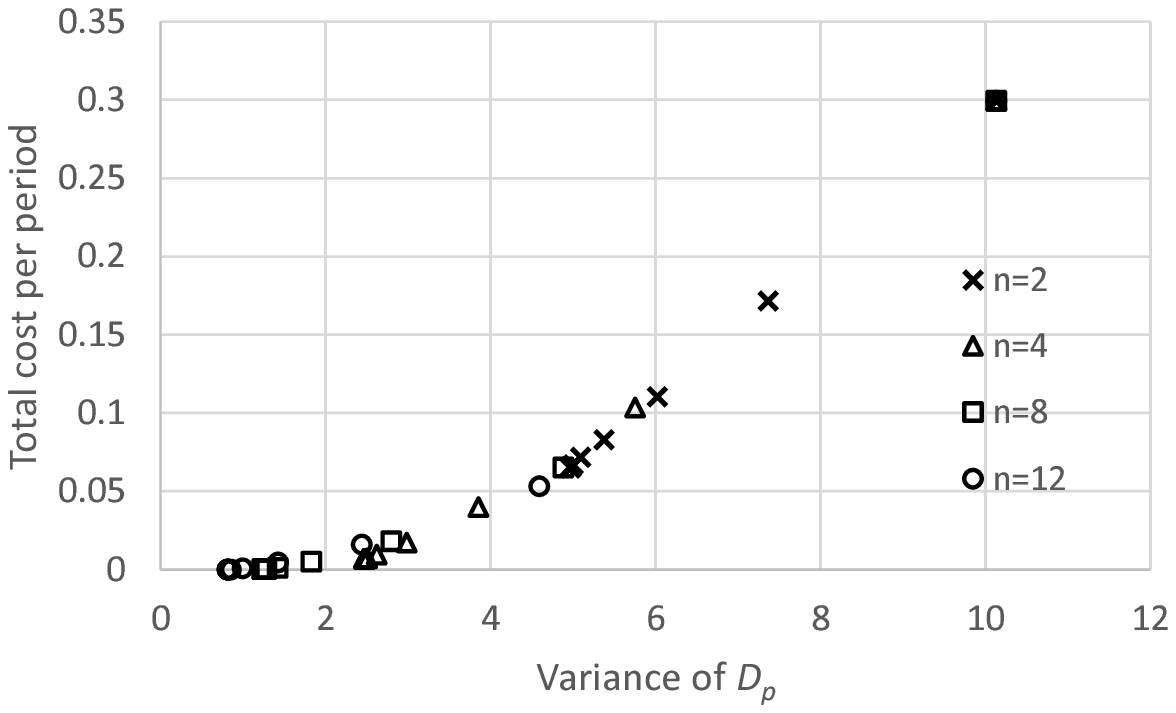}
        \caption{$m=2, q=15$ (optimal $q$)}
        \label{fig:m2q15varimpact}
   \end{subfigure}
        \begin{subfigure}[]{0.48\textwidth}
     \centering
        \includegraphics[scale=0.44]{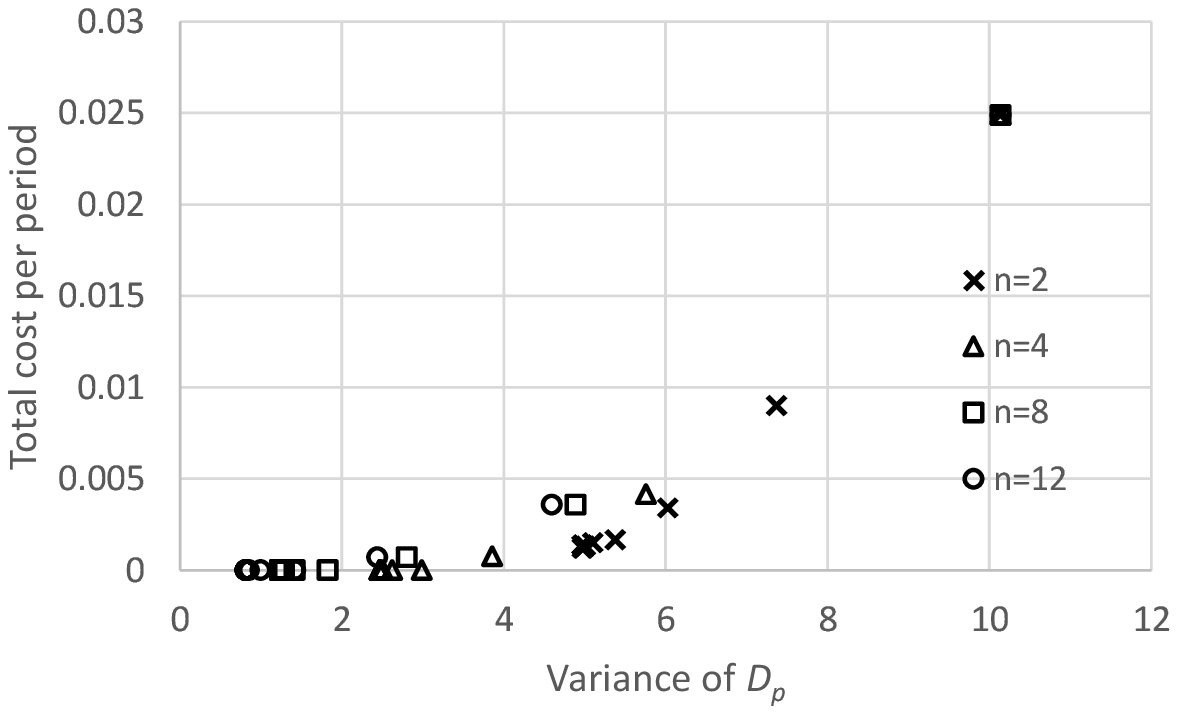}
        \caption{$m=3, q=18$ (optimal $q$)}
        \label{fig:m3q18varimpact}
   \end{subfigure}
    \begin{subfigure}[]{0.48\textwidth}
     \centering
        \includegraphics[scale=0.44]{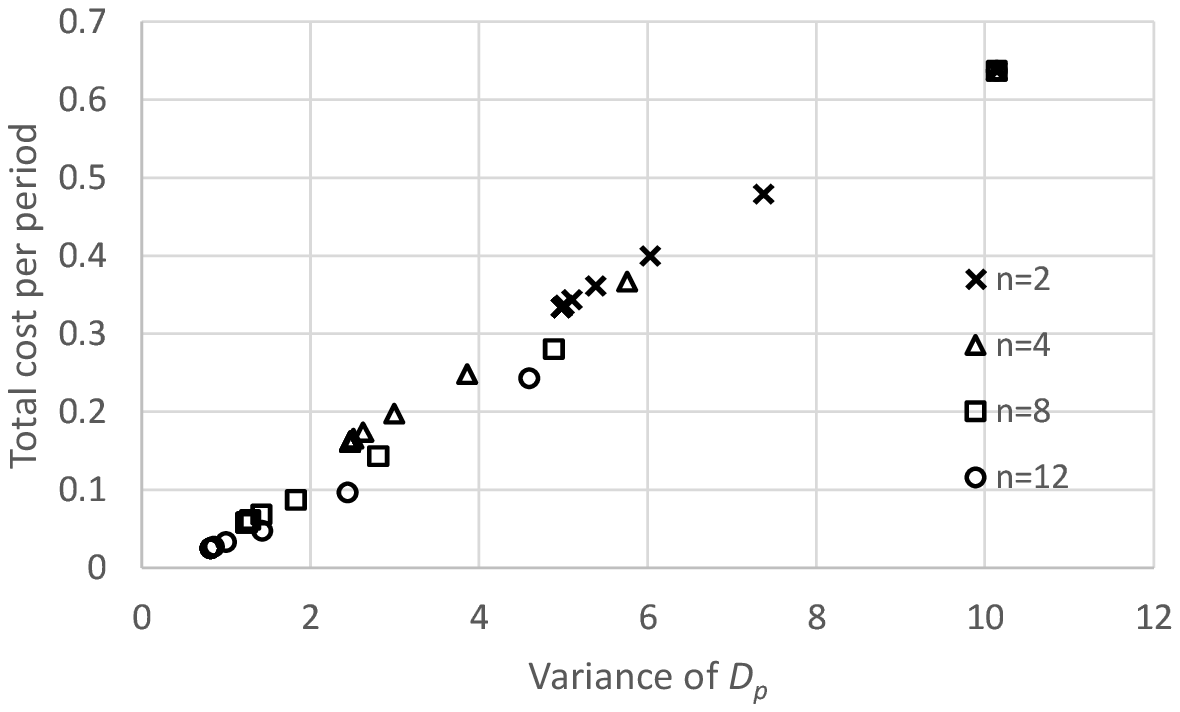}
        \caption{$m=2, q=18$ (20\% above its optimum)}
        \label{fig:m2q18varimpact}
   \end{subfigure}
    \begin{subfigure}[]{0.48\textwidth}
     \centering
        \includegraphics[scale=0.44]{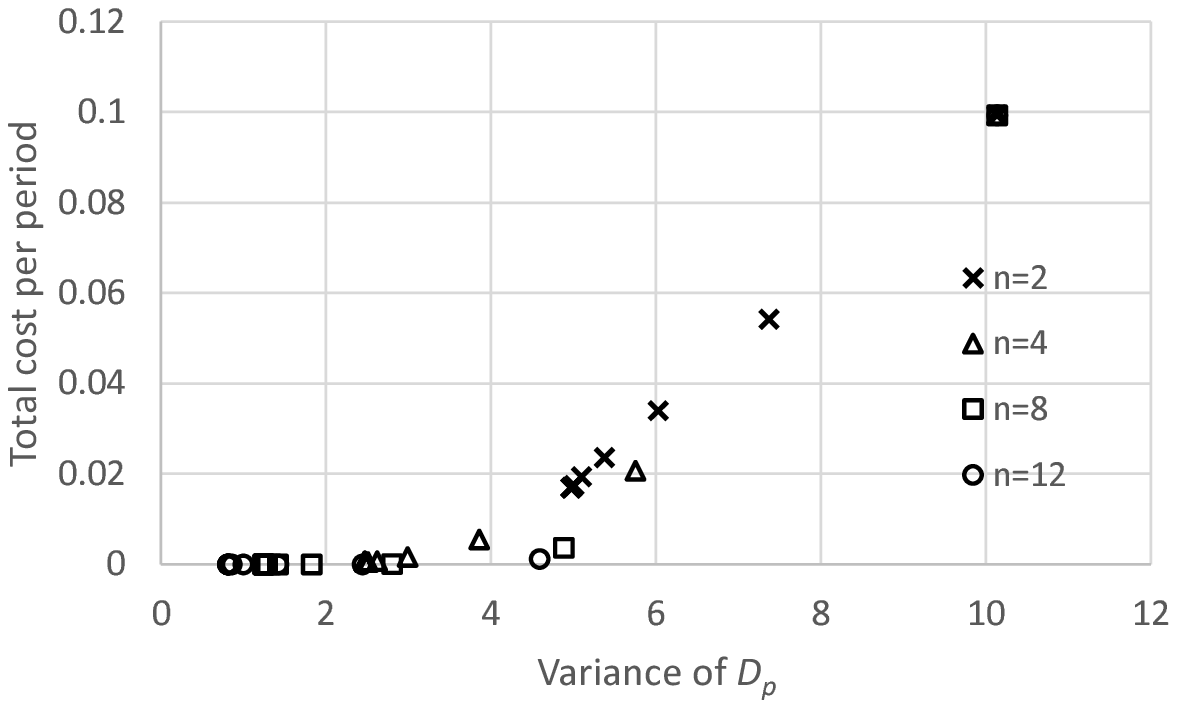}
        \caption{$m=3, q=22$ (22\% above its optimum)}
        \label{fig:m3q22varimpact}
   \end{subfigure}
\begin{flushleft}
\hspace{-.1cm} \footnotesize \emph{Notes}: Eleven plots corresponding to $p=0\%$ to $100\%$ for each $n$ are shown in each figure. Solid square dots on the top right corner of all figures correspond to the $p=0\%$ case for all $n$ (baseline case with no opaque scheme).
\end{flushleft}
\end{figure}

\subsection{\textcolor{black}{Impact of variance $\sigma_{n,p}^2$ on the expected total cost $\mathbb{E}[C]$}}
\label{5-6}

So far, our experiments indicate a strong dependence of $\mathbb{E}[C]$ on $\sigma_{n,p}^2$. To observe the relationship between $\sigma_{n,p}^2$ and $\mathbb{E}[C]$ more explicitly, we plot them in Figure~\ref{fig:vartotalcost}. In this figure, we fix $q$ around its optimal base-stock level and vary ${n\in \{2,4,8,12\}}$ and ${p\in\{0,0.1,0.2,\cdots,1\}}$ for either $m=2$ or $m=3$ cases. Figure~\ref{fig:vartotalcost} shows that $\mathbb{E}[C]$ is approximately linearly dependent on $\sigma_{n,p}^2$ irrespective of~${(n,p)}$ except when $\mathbb{E}[C]\approx 0$. This result indicates that the reduction of $\mathbb{E}[C]$ is primarily driven by the reduction of $\sigma_{n,p}^2$ until $\mathbb{E}[C]$ becomes close to zero; once we achieve $\mathbb{E}[C]\approx 0$, the further reduction of $\sigma_{n,p}^2$ does not impact $\mathbb{E}[C]$. The threshold values of $\sigma_{n,p}^2$ that achieve $\mathbb{E}[C]\approx 0$ would provide a useful tip for practitioners: it is not worthwhile to reduce ${\sigma_{n,p}^2}$ beyond the threshold since the maximum cost saving is already achieved. However, the exact threshold is hard to obtain since it depends on both $n$ and $p$; alternatively, we can approximate it by $\sigma^2_{\text{th}}$, which is easily determined by Equation~\eqref{eq:threshold} (see Definition~\ref{threshold}).


Table~\ref{tab:lowerbound} shows the total cost $\mathbb{E}[C]$ and its theoretical lower and upper bounds $\overline{C}_\text{LB}$ and $m\overline{C}_\text{LB}$, respectively, as implied by Proposition~\ref{wastage}. We evaluate $\sigma^2_{\text{th}}$ using sufficiently small $\delta$ ($\delta=0.01$). By observing the values of $\overline{C}_\text{LB}$ in Table~\ref{tab:lowerbound}, it is straightforward to identify the smallest $n$ and the largest $\sigma^2_{n,1}$ that satisfy $\overline{C}_\text{LB}\le \delta$ (and thus ${\mathbb{E}[C] \le m\overline{C}_\text{LB}\le m\delta}$): We obtain $\sigma^2_{\text{th}}=2.5$ at $q=15$ and $\sigma^2_{\text{th}}<0.83$ at $q=18$ for $m=2$; and $\sigma^2_{\text{th}}=5$ at both $q=18$ and 22 for $m=3$. These rough estimates of $\sigma^2_{\text{th}}$ are consistent with the thresholds observed in Figure~\ref{fig:vartotalcost}.
\begin{table}[H]
\begin{center}
\caption{Comparison between $\mathbb{E}[C]$ and $\overline{C}_\text{LB}$.}
\footnotesize
\begin{tabular}{ccc|rrrrr}
\noalign{\smallskip} \hline
&    & $n$   & 1           & 2          & 4          & 8         & 12 \\ \cline{3-8}
 $m$ &  $q$    & $\sigma^2_{n,1}=10/n$   & 10           & 5          & 2.5          & 1.25         & 0.83 \\ \hline
$2$ & \multicolumn{1}{r}{15} & $\mathbb{E}[C]$ & 0.2993 & 0.0673 & 0.0067 & 0.0002 & 0.0000 \\
&  & $\overline{C}_\text{LB}$ & 0.2287 & 0.0465 & \underline{0.0042} & 0.0001 & 0.0000 \\
&  & $2\overline{C}_\text{LB}$ & 0.4574 & 0.0930 & 0.0084 & 0.0002 & 0.0000 \\ \cdashline{2-8}
   & \multicolumn{1}{r}{18}   & $\mathbb{E}[C]$ & 0.6365 & 0.3455 & 0.1610 & 0.0577 & 0.0249 \\
&  & $\overline{C}_\text{LB}$ & 0.4759 & 0.2434 & 0.1080 & 0.0363 & 0.0156 \\
&  & $2\overline{C}_\text{LB}$ & 0.9519 & 0.4868 & 0.2161 & 0.0726 & 0.0311 \\ \hline
$3$ & \multicolumn{1}{r}{18} & $\mathbb{E}[C]$ & 0.0249 & 0.0006 & 0.0000 & 0.0000 & 0.0000 \\
&  & $\overline{C}_\text{LB}$ & 0.0183 & \underline{0.0005} & 0.0000 & 0.0000 & 0.0000 \\
&  & $3\overline{C}_\text{LB}$ & 0.0549 & 0.0016 & 0.0000 & 0.0000 & 0.0000 \\ \cdashline{2-8}
   & \multicolumn{1}{r}{22}   & $\mathbb{E}[C]$ & 0.0993 & 0.0166 & 0.0006 & 0.0000 & 0.0000 \\
&  & $\overline{C}_\text{LB}$ & 0.0469 & \underline{0.0065} & 0.0003 & 0.0000 & 0.0000 \\
&  & $3\overline{C}_\text{LB}$ & 0.1407 & 0.0194 & 0.0008 & 0.0000 & 0.0000  \\\hline
\end{tabular}
\label{tab:lowerbound}
\begin{flushleft}
\hspace{-0.1cm} \footnotesize \emph{Notes}: $c_v=0.32 \;(\lambda=10), p=1$. $q=15$ is optimal when $m=2$. $q=18$ is optimal when $m=3$. The underlined values specify that $\overline{C}_\text{LB}$ is the largest value that goes below ${\delta=0.01}$.
\end{flushleft}
\end{center}
\normalsize
\end{table}

\section{\textcolor{black}{Discussion and managerial insights}}
\label{sec6}
\textcolor{black}{In this section, we discuss various issues and managerial insights of our opaque scheme in detail. For readers' convenience, we list key questions one by one below.}

\subsection{\textcolor{black}{Benefit of our opaque scheme}}
\textcolor{black}{Our opaque scheme aims to provide a new method to help retailers better manage their inventory systems when suffering from large shortages and wastages. Our opaque scheme using BPD is effective for perishables: It reduces the variability of product demands, reduces shortages and wastages of perishables, and achieves a lower total cost (a higher cost saving). We provide rules of thumb to determine necessary parameters for the implementation of our opaque scheme. Customers also benefit from our opaque scheme. An opaque item for customers is simply an added option; if they do not want to purchase an opaque product, they do not need to. However, if they do purchase an opaque product, they may benefit from a discount or reward such as free/reduced shipping fee. Finally, and most importantly, the society would benefit from the opaque scheme for a potentially large aggregate reduction of food wastages. The opaque scheme, if implemented by many retail stores, could help the society to contribute to the Sustainable Development Goals (SDG) Target 12.3: ``By 2030, halve per capita global food waste at the retail and consumer levels and reduce food losses along production and supply chains, including post-harvest losses'' \cite{goals2021goal}.}

\subsection{\textcolor{black}{Recommendations for managers}}
\textcolor{black}{When product demands are highly variable, managers could first check the threshold number of non-opaque products (Definition~\eqref{threshold}) needed to achieve the maximum total cost saving. Assuming that the number of non-opaque products available for them is less than this threshold, retail store managers may next refer to the approximate relative variance (Corollary~\ref{relativevariance}) to obtain the approximate $p$ necessary to achieve most of the benefit provided by the opaque scheme. For example, we show in Figure~\ref{fig:cv} that 80\% of the benefit of the opaque scheme is achieved by setting $p=0.6 c_v$, or 90\% of the benefit of opaque scheme is achieved by setting $p=0.8 c_v$, where the benefit is defined as the reduction of variance, leading to the expected total cost savings. Since $c_v$ is often much less than 1 (for example, $c_v$ ranges from 0.1 to 0.5 in \cite{nandakumar1993near}), a smaller $p$ is sufficient to capture most of the benefit of the opaque scheme.}

\subsection{\textcolor{black}{Ease of implementation in practice}}
\textcolor{black}{The key factor that needs to be warranted is the anonymity of an opaque product: its specification must be opaque. Thus, a traditional in-person sales (where customers can easily see what's sold as an opaque product) is less suitable for an opaque selling scheme (if not impossible). This situation has drastically changed due to the rapid rise of mobile apps (E-commerce via internet enabled phones). Customers make orders and complete payments through their mobile apps from within the store or outside of the store. After orders are made via mobile apps, customers pick up ordered items (including opaque products) from staff without going through a cash register. This scheme is gaining popularity since it saves time for both customers and staff members.  An opaque sales can be implemented with almost no extra cost if stores already adopt mobile apps.}
\subsection{\textcolor{black}{Obstacles for implementation}}
\textcolor{black}{Although our BPD opaque scheme is simple and easy to implement in most cases, there are some cases that make the implementation of BPD challenging. First, we require some portion of customers to be indifferent among non-opaque products. If all customers have a strong preference to purchase a specific (non-opaque) product or if only a few switch from non-opaque to opaque products when the discount for an opaque product is huge, opaque schemes do not work well. The second, possibly more serious, obstacle for the opaque scheme is the positive correlation among all non-opaque product demands. If strong positive correlations exist among non-opaque product demands, an opaque demand (pooled demand) cannot effectively reduce the variability of non-opaque product demands. This same issue is called systematic risk (or market risk) in the insurance and finance fields; a risk pooling scheme does not work in this situation.}

\subsection{\textcolor{black}{Limitations of our opaque scheme}}
\label{limitation}
\textcolor{black}{The current study has many simplifying assumptions: scaled Poisson non-opaque product demands, i.i.d. non-opaque product demands, an constant probability to shift from each non-opaque product to an opaque product, and a constant total customer demand before/after the implementation of the opaque scheme. The inventory model we consider is also restricted. We only consider the zero lead time case following the opaque scheme literature (zero lead time is assumed ``due to the notorious complexity of multi-item inventory management'' \cite{elmachtoub2019value}). In fact, the joint impact of product demand variability and positive lead time on the total cost is difficult to incorporate in the analysis of opaque schemes. We also assume a fixed total holding cost, which includes a negligible variable (per-unit) holding cost compared to shortage and wastage costs; this assumption is reasonable when a facility cost (including utility costs such as electricity and heating costs) takes a large portion of the total holding cost. Finally, we simplify the analysis in this paper by limiting our discussion only on BPD; however, in practice, a hybrid of BPD and BPI may be a better alternative. In fact, many extensions to our model are necessary to further investigate the effectiveness of opaque schemes. We will address these extensions in the future study.}

\section{Conclusions}
\label{sec7}
\textcolor{black}{In this paper we studied the effectiveness of an opaque scheme in perishable inventory systems. Using our $(n,p)$ opaque scheme following the balancing policy on demand (BPD), we showed that store managers could reduce the variability of orders for perishable products they sell, therein reducing the expected total cost caused by shortages and wastages. Our analysis has also revealed the insight into the origins of the effectiveness of our opaque scheme. Previously, only numerical experiments showed that an opaque scheme requires only a smaller proportion of customers to switch from purchasing a non-opaque product to an opaque product \cite{elmachtoub2019value}. However, our analysis (specifically, Equation~\eqref{sigma-rel}) revealed that the necessary proportion can be small as long as the coefficient of variation of product demands is small. We confirmed the accuracy of the formula in numerical experiments. We also provided a convenient indicator, the threshold variance, that achieves the maximum total cost savings from the opaque scheme; this information helps practitioners to decide whether they want to implement our opaque scheme, and if so, which combination of parameters should be used.}

\textcolor{black}{This study shows that our opaque scheme can effectively reduce the variability of non-opaque product demands if implemented properly, leading to smaller shortages and wastages in perishable inventory systems. Moreover, due to many simplifying assumptions we make, we obtain simple analytical formulas and rules of thumb for practitioners who are considering implementing our opaque scheme. However, it requires further investigation to accommodate real-world complexities. With this study of an opaque scheme for perishable inventory systems, we hope to contribute to the reduction of food wastage, one of the most important sustainable development goals, as well as intrigue the readers into investigating this promising idea further.}

\bibliography{abandonment-ref3} 
\newpage
\appendix
\appendixtitleon
\begin{appendices}{\textbf{\large Appendix}}
\setcounter{table}{0} 
\renewcommand{\thetable}{A\arabic{table}} 
\section{Notation}
\label{notation}
Table~\ref{tab:notation} shows the notation used in our study.
\begin{table}[H]
\begin{center}
\caption{Summary of notations}
\small
\begin{tabular}{c|l}
\noalign{\smallskip} \hline
 Notation  & \hspace{10 mm} Description        \\ \hline 
$n$ & number of non-opaque products  \\ 
$i$ & index for non-opaque product $i \in [1,n]$  \\
$p_i$ & probability to switch from non-opaque product $i$ to opaque product  \\
$p$ & $p_i$ when all $p_i$ are identical  \\
$\boldsymbol{p}$ & probability vector representing all $p_i$  \\
$D_0^i$ & original demand for non-opaque product $i$ per period\\
$D_0$ & $D_0^i$ when all $D_0^i$ are i.i.d.\\
$\lambda$ & base Poisson rate for $D_0$\\
$\mu^i$ & average original demand for non-opaque product $i$ (expected value of $D_0^i$) \\
$\mu$ & $\mu^i$ when all $\mu^i$ are identical\\
$X_p^i$ & intermediate demand for non-opaque product $i$  per period \\
$X_p^0$ & intermediate demand for opaque product per period\\
$D_p^i$ & adjusted demand for non-opaque product $i$ per period \\
$D_p$ & $D_p^i$ when all $D_0^i$ are i.i.d.\\
$\sigma^2$ & average variance of original demands $D_0^i, \forall i$ \\
$\sigma^2_{n,p}$ & average variance of adjusted demands $D_p^i, \forall i$ \\
$\sigma_{rel}^2(p)$ & relative variance\\
$E[S]$ & expected number of shortages per period\\
$E[W]$ & expected number of wastages per period\\
$r$ & shortage cost per item (adjusted by the item cost) \\
$\theta$ & wastage cost per item (adjusted by the item cost)\\
$E[C]$ & expected total system cost per period\\
$F_{Y}(\cdot)$ & cumulative distribution function (CDF)\\
$P_{Y}(\cdot)$ & probability mass function (PMF)\\ \hline
\end{tabular}
\label{tab:notation}
\end{center}
\normalsize
\end{table}

\section{Proofs}
\label{proofs}

\begin{proof}[Proof of Lemma \ref{bound}.]
The first inequality, $\sigma^2 \ge {\sigma_{n,{\boldsymbol{p}}}^2}$, holds as the opaque scheme is only used to reduce ${\sigma_{n,{\boldsymbol{p}}}^2}$; the equality holds when $\boldsymbol{p}=\boldsymbol{0}$. The second inequality, $ {\sigma_{n,{\boldsymbol{p}}}^2}\ge \frac{\sigma^2}{n}$, holds since
\small
\begin{equation*}
\begin{aligned}
{\sigma_{n,{\boldsymbol{p}}}^2}&=\frac{1}{n}\sum_i var(D_{\boldsymbol{p}}^i)=\frac{1}{n}\sum_i \mathbb{E}[(D_{\boldsymbol{p}}^i-\mu^i)^2]=\mathbb{E}\left[\frac{1}{n}\sum_i \left(D_{\boldsymbol{p}}^i-\mu^i\right)^2\right] \\
&=\mathbb{E}\left[\left(\frac{1}{n}\sum_i (D_{\boldsymbol{p}}^i-\mu^i)\right)^2 +\frac{1}{n^2}\mathop{\sum \sum}_{i<j} \left((D^i_p-\mu^i)-(D^j_p-\mu^j)\right)^2 \right] \\
&\ge \mathbb{E}\left[\left(\frac{1}{n}\sum_i (D_{\boldsymbol{p}}^i-\mu^i)\right)^2\right]  =\mathbb{E}\left[\left(\frac{1}{n}\sum_i D_{\boldsymbol{p}}^i-\frac{1}{n}\sum_i \mu^i\right)^2 \right] \\
&=\mathbb{E}\left[\left(\frac{1}{n}\sum_i D_0^i-\frac{1}{n}\sum_i \mu^i\right)^2 \right]=var\left(\frac{1}{n} \sum_i D_0^i\right)=\frac{\sigma^2}{n},
\end{aligned}
\end{equation*}
\normalsize
where an equality holds when $D^i_p-\mu_i=D^j_p-\mu_j$ for all $i$ and $j$, which is realized when $\boldsymbol{p}=\boldsymbol{1}$.
\end{proof}

\begin{proof}[Proof of Proposition \ref{varepsilon}.]
The following identity relationship exists when $n=2$: $X_p^1+X_p^2+X^0_p=D_p^1+D_p^2=D_0^1+D_0^2$, where $X_p^1, X_p^2, X^0_p$ are independent scaled Poisson random variables, and $D_0^1$ and $D_0^2$ are independent scaled Poisson random variables. Thus, $\sigma_T^2=var(D_p^1+D_p^2)=var(D_0^1\pm D_0^2)$. Since $D_p^1=D_p^2$ when $p=1$, we have
\small
\begin{equation*}
\begin{aligned}
\sigma_{n,p}^2&=\frac{var(D_p^1)+var(D_p^2)}{2}=\frac{var(D_p^1+D_p^2)}{4}+\frac{var(D_p^1-D_p^2)}{4}=\frac{\sigma_T^2}{4}+\frac{var(D_p^1-D_p^2)}{4},\\
\sigma_{n,0}^2&=\frac{\sigma_T^2}{2},\; \sigma_{n,1}^2=\frac{\sigma_T^2}{4},\;
\sigma_{rel}^2=\frac{\sigma_{n,p}^2-\sigma_{n,1}^2}{\sigma_{n,0}^2-\sigma_{n,1}^2}=\frac{var(D_p^1-D_p^2)}{\sigma_T^2}.
\end{aligned}
\end{equation*}
\normalsize
Note that $D_p^1$ and $D_p^2$ are not independent except when $p=0$. Thus, to find $var(D_p^1-D_p^2)$, we separate the entire state space into three disjoint sets:
\small
\begin{equation*}
\begin{aligned}
A_1&=\{X_p^1, X_p^2, X^0_p|\:X_p^1 - X_p^2 > X^0_p\}=\{X_p^1, X_p^2, X^0_p|\:T>0\}\\
A_2&=\{X_p^1, X_p^2, X^0_p|\: |X_p^1 - X_p^2|\le X^0_p\},\;
A_3=\{X_p^1, X_p^2, X^0_p|\:X_p^1 - X_p^2 < -X^0_p\}
\end{aligned}
\end{equation*}
\normalsize
Here, observe that $A_1$ and $A_3$ are identical when $X_p^1$ and $X_p^2$ are flipped, and also that ${D_p^1-D_p^2=T}$ given $A_1$ and ${D_p^1-D_p^2=0}$ given $A_2$. Therefore we have
\small
\begin{equation*}
\begin{aligned}
var(D_p^1-D_p^2|A_1)&=var(D_p^1-D_p^2|A_3)=var(T|T>0),\; var(D_p^1-D_p^2|A_2)=0,\\
\mathbb{E}[D_p^1-D_p^2|A_1]&=-\mathbb{E}[D_p^1-D_p^2|A_3]=\mathbb{E}[T|T>0],\; \mathbb{E}[D_p^1-D_p^2|A_2]=0, \\
\mathbb{E}[D_p^1-D_p^2]&=0,\; Pr(A_1)=Pr(A_3)=Pr(T>0).\\
\end{aligned}
\end{equation*}
\normalsize
We can now represent $var(D_p^1-D_p^2)$ using $T$ as follows:
\small
\begin{equation*}
\begin{aligned}
var(D_p^1-D_p^2)&=\mathbb{E}[var(D_p^1-D_p^2|A)]+var(\mathbb{E}[D_p^1-D_p^2|A]) \\
&=\sum_i var(D_p^1-D_p^2|A_i)Pr(A_i)+\sum_i (\mathbb{E}[D_p^1-D_p^2|A_i])^2 Pr(A_i)\\
&=2 \left[var(T|T>0)+(\mathbb{E}[T|T>0])^2\right] Pr(T>0)\\
&=2\mathbb{E}[T^2|T>0] Pr(T>0),
\end{aligned}
\end{equation*}
\normalsize
from which we can derive the representation of $\sigma_{rel}^2$.
\end{proof}

\begin{proof}[Proof of Corollary \ref{relativevariance}.]
The difference between two independent Poisson random variables (or scaled variables of them) is known to follow a Skellam distribution, which is approximately represented by a normal distribution. Since $X_p^1$ and $X_p^2+X^0_p$ are independent scaled Poisson random variables, the difference $T$ approximately follows a normal distribution: $T \propto N(-2p\mu,\frac{2\mu^2}{\lambda})$, whose mean is $\mu_T=-2p\mu$ and standard deviation is $\sigma_T=\mu\sqrt{\frac{2}{\lambda}}$. The z-score at $T=0$ is ${\alpha=\frac{-\mu_T}{\sigma_T}=p\sqrt{2\lambda}=\frac{p}{c_v}\sqrt{2}}$. Let also ${Z=1-\Phi(\alpha)=\Phi(-\alpha)\: (\approx Pr(Z>0))}$. The distribution of $T$ given~${T>0}$ is approximated by a truncated normal distribution, whose mean and variance are
\small
\begin{equation*}
\begin{aligned}
\mathbb{E}[T|T>0]&\approx \mu_T+\sigma_T \frac{\phi(\alpha)}{Z}=\sigma_T\left(-\alpha+ \frac{\phi(\alpha)}{Z}\right)\\
var(T|T>0)&\approx \sigma_T^2 \left(1+\alpha \frac{\phi(\alpha)}{Z}  - \left(\frac{\phi(\alpha)}{Z} \right)^2 \right).\\
\end{aligned}
\end{equation*}
\normalsize
Using these relationships with Proposition~\ref{varepsilon}, we obtain
\small
\begin{equation*}
\begin{aligned}
\sigma_{rel}^2&=\frac{2 \left[var(T|T>0)+(\mathbb{E}[T|T>0])^2\right] Pr(T>0)}{\sigma_T^2}\\
&\approx 2(1+\alpha^2)\Phi(-\alpha)-2\alpha\phi(\alpha).
\end{aligned}
\end{equation*}
\normalsize
\end{proof}

\begin{proof}[Proof of Lemma \ref{involution}.]
We use $\sigma_{n,0}^2$ to represent $\sigma_{n,p}^2$ and $\sigma_{n,1}^2$, from which we obtain~$\sigma_{rel}^2(p)$.
\small
\begin{equation*}
\begin{aligned}
{\sigma_{n,0}^2}&=\frac{1}{n} \sum_i var(D_0^i)=\frac{1}{n} var\left(\sum_i D_0^i\right)=\frac{1}{n} var\left(\sum_i D_p^i\right)=\frac{1}{n} \sum_{i=1}^n \sum_{j=1}^n cov(D_p^i,D_p^j)\\
&=\frac{1}{n}\sum_i^n  var(D_p^i) + \frac{2}{n}\mathop{\sum \sum}_{1 \le i<j \le n}  cov(D_p^i,D_p^j) =\left(1+(n-1)\rho_p\right)\sigma_{n,p}^2 \,,\\
\sigma_{n,1}^2&=\frac{1}{n} \sigma_{n,0}^2\,,\\
\sigma_{rel}^2(p)&=\frac{\sigma_{n,p}^2 - \sigma_{n,1}^2}{\sigma_{n,0}^2 - \sigma_{n,1}^2}=\frac{\frac{\sigma_{n,0}^2}{1+(n-1)\rho_p} - \frac{\sigma_{n,0}^2}{n}}{\sigma_{n,0}^2 - \frac{\sigma_{n,0}^2}{n}}=\frac{1-\rho_p}{1+(n-1)\rho_p}\,.
\end{aligned}
\end{equation*}
\normalsize
\end{proof}

\begin{proof}[Proof of Proposition \ref{wastage}.]
We first prove Equation \eqref{expshortage}. By definition, $\mathbb{E}[S]=\mathbb{E}[D_1-q]^+$. Also note that $D_1=\frac{\mu}{n\lambda}Y_{n\lambda}$ and $P_{Y_{\gamma}}(s)=\frac{\gamma^s e^{-\gamma}}{s!}$. It is straightforward to obtain
\small
\begin{align*}
\mathbb{E}[S] &= \mathbb{E}\left[D_1-q\right]^+=\mathbb{E}\left[\frac{\mu}{n\lambda}Y_{n\lambda}-q\right]^+=\frac{\mu}{n\lambda}\mathbb{E}\left[Y_{n\lambda}-\frac{n\lambda q}{\mu}\right]^+\\
&=\frac{\mu}{n\lambda}\sum_{y=s+1}^\infty \left(y-\frac{n\lambda q}{\mu}\right) P_{Y_{n\lambda}}(y)=\frac{\mu}{n\lambda}\left[\sum_{y=s+1}^\infty y P_{Y_{n\lambda}}(y)-\frac{n\lambda q}{\mu}\sum_{y=s+1}^\infty  P_{Y_{n\lambda}}(y)\right]\\
&=\frac{\mu}{n\lambda}\left[n\lambda \sum_{y=s}^\infty P_{Y_{n\lambda}}(y)-\frac{n\lambda q}{\mu}\sum_{y=s+1}^\infty  P_{Y_{n\lambda}}(y)\right]\\
&=\mu \left(1-F_{Y_{n\lambda}}(s)+P_{Y_{n\lambda}}(s)\right)-q\left(1-F_{Y_{n\lambda}}(s)\right)\\
&=(\mu-q)\left(1-F_{Y_{n\lambda}}(s)\right)+\mu P_{Y_{n\lambda}}(s).
\end{align*}
\normalsize
Next, we prove Equation \eqref{expwastage}. The proof of $\mathbb{E}[W] \geq \mathbb{E}\left[\frac{q}{m}-\overline{D}_1\right]^+$ is presented in \cite{chazan1977markovian}; this lower bound is obtained by evaluating the per-period (i.e., $1/m$ of the) expected wastage of the first $m$ periods assuming that all inventory in the beginning of period 1 is fresh. In contrast, the upper bound is the expected wastage of the $m\textsuperscript{th}$ period assuming again that all inventory in the beginning of period 1 is fresh, since the amount of the oldest units is largest in the beginning of the $m\textsuperscript{th}$ period. We omit the remaining part of the derivation of Equation~ \eqref{expwastage} since it is almost identical to the derivation of Equation~\eqref{expshortage} except we use $Y_{nm\lambda}$ instead of $Y_{n\lambda}$.
\end{proof}
\end{appendices}
\end{document}